\documentclass[12pt,a4paper,reqno]{amsart}
\usepackage{pifont}
\usepackage[english]{babel}
\usepackage[utf8]{inputenc}
\usepackage[IL2]{fontenc}
\usepackage[margin=0.8in]{geometry}
\usepackage{palatino}
\usepackage{amsmath,amssymb,amsthm}
\usepackage{url}
\usepackage{float}
\usepackage{verbatim}
\usepackage{marvosym}
\usepackage{afterpage}
\usepackage[ruled,vlined]{algorithm2e}
\usepackage[colorlinks,linkcolor=black,urlcolor=black,citecolor=green]{hyperref}
\usepackage{tikz}
\usepackage{tikz-cd}
\usetikzlibrary{matrix,calc,arrows,decorations.markings,positioning,shapes.geometric}

\hypersetup{}

\theoremstyle{plain}
\newtheorem{thm}{Theorem}[section]
\newtheorem{app}[thm]{Addendum}
\newtheorem{alg}[thm]{Algorithm}
\newtheorem{lem}[thm]{Lemma}
\newtheorem{prop}[thm]{Proposition}
\newtheorem{cor}[thm]{Corollary}
\newtheorem{thma}{Theorem}

\newtheorem{cora}[thma]{Corollary}

\newtheorem{con}[thm]{Convention}
\theoremstyle{definition}
\newtheorem{defn}[thm]{Definition}
\newtheorem*{exmp}{Example}
\theoremstyle{remark}
\newtheorem*{rem}{Remark}

\numberwithin{equation}{section}
\newcommand{\Qbb}{\mathbb{Q}}

\newcommand{\Zbb}{\mathbb{Z}}
\newcommand{\Nbb}{\mathbb{N}}

\newcommand{\id}{\mathrm{id}}

\newcommand{\Isoef}{\operatorname{Iso}^{\operatorname{ef}}}
\newcommand{\Iso}{\operatorname{Iso}}
\newcommand{\iso}{\operatorname{iso}}
\newcommand{\isoef}{\operatorname{iso}^{\operatorname{ef}}}
\newcommand{\aut}{\operatorname{aut}}
\newcommand{\autef}{\operatorname{aut}^{\operatorname{ef}}}
\newcommand{\Autef}{\operatorname{Aut}^{\operatorname{ef}}}
\newcommand{\Aut}{\operatorname{Aut}}
\newcommand{\GL}{\operatorname{GL}}

\newcommand{\ev}{\operatorname{ev}}
\newcommand{\EM}{M^{\operatorname{ef}}}
\newcommand{\EC}{C^{\operatorname{ef}}}
\newcommand{\ED}{D^{\operatorname{ef}}}
\newcommand{\econ}{\operatorname{cone}^{\operatorname{ef}}}
\newcommand{\E}[1]{{#1}^{\operatorname{ef}}}
\newcommand{\Free}{\operatorname{Free}}
\newcommand{\Tor}{\operatorname{Tor}}

\newcommand{\Hom}{\operatorname{Hom}}
\newcommand{\Stab}{\operatorname{Stab}}
\newcommand{\eiota}{\iota^{\operatorname{ef}}}
\newcommand{\ekappa}{\kappa^{\operatorname{ef}}}
\newcommand{\erho}{\rho^{\operatorname{ef}}}
\newcommand{\rank}{\operatorname{rank}}

\newcommand\algline{\noindent \hrulefill~ {\it{Algorithmic part}}~ \hrulefill}

\makeatletter
\providecommand*{\rightarrowfill@}{%
  \arrowfill@\relbar\relbar\rightarrow
}
\providecommand*{\xrightarrow}[2][]{%
  \ext@arrow 0579\rightarrowfill@{#1}{#2}%
}
\providecommand*{\twoheadrightarrowfill@}{%
  \arrowfill@\relbar\relbar\twoheadrightarrow
}
\providecommand*{\xtwoheadrightarrow}[2][]{%
  \ext@arrow 0579\twoheadrightarrowfill@{#1}{#2}%
}
\newcommand{\proofpart}[1]{%
  \par
  \addvspace{\medskipamount}%
  \noindent\emph{Part #1:}\par\nobreak
  \addvspace{\smallskipamount}%
  \@afterheading
}

\makeatother
\begin{document}
\title[Are two finite $H$-spaces homotopy equivalent?]{Are two $H$-spaces homotopy equivalent?\\ An algorithmic view point}
\begin{abstract}
This paper proposes an algorithm that decides if two simply connected spaces represented by finite simplicial sets of finite $k$-type and finite dimension $d$ are homotopy equivalent. 
If the spaces are homotopy equivalent, the algorithm finds a homotopy equivalence between their Postnikov stages in dimension $d$. As a consequence, we get an algorithm deciding 
if two spaces represented by finite simplicial sets are stably homotopy equivalent.
\end{abstract}
\author{M\'{a}ria \v{S}imkov\'{a}}
\address{Department of Mathematics and Statistics, Masaryk University, Kotl\'{a}\v{r}sk\'{a} 2, 611 37 Brno,  Czech Republic}
\curraddr{}
\email{simkova@math.muni.cz}
\thanks{}
\keywords{Postnikov tower, Postnikov invariants, effective homological framework, algorithmic computation, homotopy equivalence problem, spaces of finite k-type, $H$-spaces}
\subjclass[2010]{55S45, 68U05, 55U10.}

\date{\today}

\dedicatory{}

\maketitle

\section{Introduction}

The question of whether two topological spaces are homotopy equivalent has motivated the deve\-lop\-ment of classical algebraic topology. It has contributed to the search for various algebraic invariants. A somewhat different approach was started in 1957 by E.~H.~Brown. In \cite{B}, he asked whether there is an algorithm that can decide if two simply connected spaces described by finite simplicial sets are homotopy equivalent. He answered it in the case that the spaces have finite homotopy groups. The algorithm uses Postnikov towers and is an exhaustive search through a finite set of combinations. In the paper \cite{NW} from the end of the 90s, A.~Nabutovsky and S.~Weinberger outlined how to decide if two simply connected PL-manifolds of dimension at least five are homotopy equivalent and then even more if they are homeomorphic or diffeomorphic. Their idea for homotopy equivalence is based again on Postnikov towers, Postnikov invariants, and additionally on algorithms by Grunewald and Segal (\cite{GS}) working with arithmetic subgroups of linear algebraic groups. However, their idea was only briefly sketched. Since that time, significant progress has been made in the algorithmic construction of Postnikov towers, see \cite{sphere}, \cite{poly} and \cite{aslep}. This has allowed us to approach the problem in more detail. At first, we establish a necessary and sufficient condition to lift a given homotopy equivalence between $n$-stages of Postnikov towers of two spaces into a homotopy equivalence between $(n+1)$-stages. This enables us to decide algorithmically about homotopy equivalence if both spaces satisfy certain finiteness conditions.

It turns out that the actions of certain groups on the set of cohomological classes play a crucial role in solving the problem of homotopy equivalence.
In general,  both the groups and the sets are infinite. Here we will use suitable group algorithms to solve the question of homotopy equivalence in the case of infinite (but finitely generated) groups and finite sets. This situation occurs for spaces of finite $k$-type. These are spaces with torsion Postnikov classes through a specific finite dimension. 
The following spaces belong among them 
\begin{itemize}
    \item $H$-spaces, especially Lie groups, topological groups, and simplicial groups,
    \item $m$-connected spaces with dimension at most $2m$,
    \item $H$-spaces modulo the class of finite groups (see the definition in Section 6),
    \item spaces rationally homotopy equivalent to products of Eilenberg-McLane spaces.
\end{itemize}
Our restriction to spaces of finite $k$-type corresponds to a recent result by Manin (\cite{FM}, Theorem A) on the extendability of maps.  

This paper's main result is:
\begin{thma}
Let $X$ and $Y$ be simply connected finite simplicial set of dimension $d$. Suppose that they are of finite $k$-type through dimension $d$. Then the question of whether $X$ and $Y$ are homotopy equivalent is algorithmically decidable. 
If the spaces are homotopy equivalent, we can find a homotopy equivalence $f\colon X_d\to Y_d$ between their Postnikov stages in dimension $d$. 

Moreover, we can algorithmically find all group generators of the group $\Aut(X_d)$ of  self homotopy equivalences $X_d$ up to homotopy.
\end{thma}

As a consequence of this theorem, we get:
\begin{cora}
Let $X$ and $Y$ be finite simplicial sets. Then the question, if $X$ and $Y$ are stably homotopy equivalent, is algorithmically decidable.
\end{cora}

\subsection*{Plan of the paper}
To prove our main result, we use the fact that simply connected spaces of dimension $d$ are homotopy equivalent if and only if they are homotopy equivalent through $d$-stages of their Postnikov towers. Thus, we try to find a homotopy equivalence from the lower to higher stages of Postnikov towers.
Sections 2 and 3 have preparatory character and summarize essential notions from simplicial sets, effective homology framework, and Postnikov towers.
Section 4 gives the necessary and sufficient conditions for lifting homotopy equivalence between given stages of Postnikov towers by one step higher. In Section 5, we introduce the notion of effective homotopy equivalence between Postnikov stages. These are homotopy equivalences that can be constructed algorithmically by induction. It turns out that the effective self-equivalences of a Postnikov stage with the composition operation form a group that behaves well. Notably, the homotopy class of any homotopy equivalence can be represented as the homotopy class of an effective homotopy equivalence. In the second part of this section, the necessary and sufficient conditions are rewritten in terms of effective homotopy equivalences.

Section 6 is an excursion into algorithmic group theory. If we have a group with a finite list of its generators and an action of this group on a finite set, there are simple algorithms that compute a given element's orbit and its stabilizer. These algorithms are much more straightforward than the algorithms by Grunewald and Segal. In Section 7, the notion of finite $k$-type is introduced (it means that Postnikov invariants are torsion elements), which enables us to use the mentioned group algorithms acting on a torsion part of a certain cohomology group, i.e. on a finite set. In the general case, the set under the action is the whole infinite cohomology group. Section 8 completes the description of the algorithm for the main result. In the last section, another class of spaces that admits algorithmic decisions about homotopy equivalence is characterized. 

\smallskip
In further work, we want to obtain an algorithm for all simply connected finite simplicial sets. To resolve it, we need more sophisticated algorithms for orbits and stabilizers of finitely generated groups and a suitable way to apply them. We believe that the notion of effective homotopy equivalence introduced here will provide the right environment to create such an algorithm.

\section{Preliminaries on simplicial sets with effective homology}

We will work with simplicial sets instead of simplicial complexes since they are more powerful and flexible. For basic concepts on simplicial sets, we refer to comprehensive sources \cite{M, C, GJ}. In this section, we show briefly what we need to use an algorithmic approach to the homotopy theory of simplicial sets.

\subsection*{Simplicial sets}
The standard $k$-dimensional simplex is denoted $\Delta^k$ and its $n$-simplices (non-degenerate) are $(n+1)$-element subsets of its vertex set $\{0,1,\dots,k\}$.
Let $\pi$ be an abelian group.
In the simplicial set context, the Eilenberg-MacLane simplicial set $K(\pi,n)$ is defined through its \emph{standard minimal} model in which $k$-simplices are given by cocycles
\[K(\pi,n)_k=Z^n(\Delta^k,\pi).\]
It will be essential that $K(\pi,n)$ is a simplicial group in the mentioned representation and hence a Kan complex. Similarly, we define a simplicial set $E(\pi,n)$ where its $k$-simplices are given by cochains
\[E(\pi,n)_k=C^n(\Delta^k,\pi).\]
The previous definitions lead to a natural principal fibration known as the \emph{Eilenberg-MacLane fibration} $\delta\colon E(\pi,n)\to K(\pi,n+1)$ for $n\geq 1$. A further powerful property of that fibration is its \emph{minimality}.  Recall that \emph{minimal} fibrations are stable under pullbacks and compositions. For the definition and further properties see \cite[I.10]{GJ}.\\

Furthermore, we remind two standard constructions of simplicial sets - a mapping cylinder and a mapping cone. For a simplicial map $f\colon X\to Y$ the simplicial \emph{mapping cylinder} is a simplicial set  $\operatorname{cyl}(f)=(X\times\Delta^1\cup Y)/\sim$ where the equivalence relation $\sim$ is induced by $(x,0)\sim f(x)$ for all $x\in X$. If we identify $X$ with $X\times\{1\}$ then the simplicial \emph{mapping cone} of $f$ is $\operatorname{cone}(f)=\operatorname{cyl}(f)/X$.

\medskip
\algline

\subsection*{Effective homology}
Here we look at the basic notions of the effective homology framework. 
This paradigm was developed by Sergeraert and his coworkers to deal with infinitary objects, see \cite{RS} (or \cite{poly}) for more details. \\

A \emph{locally effective} simplicial set is a simplicial set whose simplices
have a specified finite encoding, and whose face and degeneracy operators are specified by algorithms.

We will work with non-negatively graded chain complexes of free abelian groups. Such a chain complex is \emph{locally effective} if elements of the graded module can be represented in a computer and the operations of zero, addition, and differential are computable.\\

In all parts of the paper where we deal with algorithms, all simplicial sets are locally effective, and all chain complexes
are non-negatively graded locally effective chain complexes of free $\Zbb$-modules.
All simplicial maps, chain maps, chain homotopies, etc., are computable.\\

An \emph{effective} chain complex is a (locally effective) free chain complex equipped
with an algorithm that generates a list of elements of the distinguished basis in any given dimension (in particular, the distinguished bases are finite in each dimension).

\begin{defn}
[\cite{RS}]
Let $(C,d_C)$ and $(D,d_D)$ be chain complexes. A triple of mappings $(f\colon C\to D,g\colon D\to C,h\colon C\to C)$ is called a \emph{reduction} if the following holds
\begin{itemize}
\item[i)] $f$ and $g$ are chain maps of degree $0$,
\item[ii)] $h$ is a map of degree 1,
\item[iii)] $fg=\id_D$ and $\id_C-gf=[d_C,h]=d_Ch-hd_C$,
\item[iv)] $fh=0$, $hg=0$ and $hh=0$.
\end{itemize}
The reductions are denoted as $(f,g,h)\colon (C,d_C)\Rightarrow (D,d_D)$.

A \emph{strong homotopy equivalence} $C\Longleftrightarrow D$ between chain complexes $C$, $D$ is the chain complex $E$ together with a pair of reductions $C\Leftarrow E\Rightarrow D$.

Let $C$ be a chain complex. We say that $C$ is \emph{equipped with effective
homology} if there is a specified strong equivalence $C\Longleftrightarrow \EC$
of $C$ with some effective chain complex $\EC$.

Similarly, we say that a \emph{simplicial set has} (or can be equipped with) \emph{effective homology} if its chain complex generated by nondegenerate simplices is equipped with  effective homology.
\end{defn}

It is clear that all finite simplicial sets have effective homology. It is essential from the algorithmic point of view that many infinite simplicial sets also have effective homology. Moreover, there is a way to construct them from the underlying simplicial sets and their effective chain complexes.

\begin{prop}[\cite{poly}, Section 3]
Let $n\geq 1$ be a fixed integer and $\pi$ a finitely generated abelian group. The standard simplicial model of the Eilenberg-MacLane space can be equipped with effective homology.

If $P$ is a simplicial set  equipped  with  effective  homology and $f\colon P\to K(\pi,n+1)$ is  computable, then  the  pullback $Q$ of $\delta\colon E(\pi,n)\to K(\pi,n+1)$ along $f$ can  be  equipped  with  effective homology.
\end{prop}

The algebraic mapping cone $\operatorname{cone}(\varphi)$ of the chain map $\varphi\colon C\to D$ between chain complexes $(C_*,d_C)$ and $(D_*,d_D)$ is defined as the chain complex $C_{*-1}\oplus D_*$ with  the differential $(x,y)\mapsto(-d_C(x),d_D(y)+\varphi(x))$.

\begin{prop}[\cite{RS}, Theorems 63 and 80]\label{ConeEf}
Let $C$ and $D$ be chain complexes with  effective  homology $\EC$ and $\ED$, respectively,  and $\varphi\colon C\to D$ be  a chain  map.  Then  the $\operatorname{cone}(\varphi)$ can be equipped with effective homology $\econ:=\EC_{*-1}\oplus \ED_*$ in such a way that the strong equivalence $\operatorname{cone}(\varphi)\Longleftrightarrow\econ$ restricts to the original string equivalence $D\Longleftrightarrow\ED$.
\end{prop}

\section{Postnikov towers}
In this section, we recall the notion of a general Postnikov tower \cite[Chapter VI]{GJ} and the construction of a functorial Postnikov tower for a given Kan complex by Moore \cite{MO} and make a comparison of both definitions. Then we describe an algorithmic construction of a Postnikov tower for simplicial sets with effective homology \cite{poly}, which we will use in the following sections.

\begin{defn}
Let $Y$ be a simplicial set. A simplicial Postnikov tower  for $Y$ is the following collection of mappings and simplicial sets organized into the commutative diagram
\begin{equation}
\begin{tikzpicture}
\matrix (m) [matrix of math nodes, row sep=2em,
column sep=2em, minimum width=2em]
{  &[2cm] Y_n\\
    &[2cm] \vdots  \\
   &[2cm] Y_1 \\
  Y &[2cm] Y_0\\
};
  \begin{scope}[every node/.style={scale=.8}]
\path[->](m-4-1) edge node[above] {$\varphi_0$} (m-4-2);
\path[->](m-4-1) edge node[above] {$\varphi_1$} (m-3-2);
\path[->](m-4-1) edge node[auto] {$\varphi_n$} (m-1-2);
\path[->](m-3-2) edge node[auto] {$p_1$} (m-4-2);
\path[->](m-2-2) edge node[auto] {$p_2$} (m-3-2);
\path[->](m-1-2) edge node[auto] {$p_n$} (m-2-2);
\end{scope}
\end{tikzpicture}\label{3d2}
\end{equation}
such that for each $n\geq0$ the map $\varphi_n\colon Y\to Y_n$
induces isomorphisms $\varphi_{n*}\colon \pi_k(Y)\to \pi_k(Y_n)$
of homotopy groups with $0\leq k\leq n$, and $\pi_k(Y_n)=0$ for $k\geq n+1$.
The simplicial set $Y_n$ is called the $n$-th Postnikov stage. 
\end{defn}

Homotopy groups of simplicial sets in terms of simplicial maps can be defined only for Kan complexes. To other simplicial sets, this notion can be extended using the geometric realization functor $|\ |$ and its right adjoint simplicial functor $S$ if we put $\pi_n(Y)=\pi_n(S|Y|)$ since the image of $S$ is in the subcategory of Kan complexes. If we suppose that $Y$ is a Kan complex, we can carry out the following construction due to Moore.

\begin{defn}
Let $Y$ be a Kan complex. For each $n\in\Nbb_0$ define an equivalence relation $\sim_n$ on the simplices of $Y$ as follows: two $q$-simplices $x,y\colon\Delta^q\to Y$
are equivalent if their restrictions on $n$-skeleton $x|_{\operatorname{sk}_n(\Delta^q)}$ and $y|_{\operatorname{sk}_n(\Delta^q)}$ are equal.
Define a simplicial set $Y(n):=Y/\sim_n$ together with induced degeneracy and face maps from $Y$. There are evident maps $p(n)\colon Y(n)\to Y(n-1)$ and $\varphi(n)\colon Y\to Y(n)$.
\end{defn}

It turns out that this collection is a Postnikov tower for $Y$. Therefore, we will call it
\emph{Moore-Postnikov tower} per \cite{GJ}. Furthermore:

\begin{prop}\label{P}
Let $Y$ be a simply connected Kan complex. The tower $\{Y(n)\}_{n\in\Nbb}$ defined above is a Postnikov tower for $Y$ and the maps $p(n)$ are fibrations. Furthermore, if $Y$ is a  simply connected minimal Kan complex then $p(n)$ are minimal fibrations and pullbacks of diagrams along a certain maps $k(n-1)$:
\begin{center}
\begin{tikzpicture}
\matrix (m) [matrix of math nodes, row sep=2em,
column sep=2em, minimum width=2em]
{  Y(n) & E(\pi_n(Y),n)  \\
  Y(n-1) & K(\pi_n(Y),n+1) \\
};
  \begin{scope}[every node/.style={scale=.8}]
\path[->](m-1-1) edge  (m-1-2);
\path[->](m-1-1) edge node[left] {$p(n)$} (m-2-1);
\path[->](m-2-1) edge node[below] {$k(n-1)$} (m-2-2);
\path[->](m-1-2) edge node[right] {$\delta$} (m-2-2);
\end{scope}
\end{tikzpicture}
\end{center}
\end{prop}

\begin{proof}
These are Theorems 2.6 and 3.26 in \cite{MO}. See also \cite{GJ}, Chapter VI, Theorem 2.5, and Corollary 5.13.
\end{proof}

The main advantage of the Moore-Postnikov tower is that it forms a functor from the category of Kan complexes into a category of towers of simplicial sets. However, we would like to construct Postnikov towers directly from simplicial sets which are not Kan complexes. A common procedure for doing this is to build up the $n$-th stage from the $(n-1)$-th stage as a pullback of the minimal Eilenberg-MacLane fibration $\delta\colon E(\pi_n(Y),n)\to  K(\pi_n(Y),n+1)$.

\begin{defn}\label{3d1}
Let $Y$ be a simply connected simplicial set. A \emph{standard Postnikov tower} is a Postnikov tower such that for all $n\geq1$: $Y_n$ is the pullback of the fibration $\delta$ along a map $k_{n-1}\colon Y_{n-1}\to K(\pi_n(Y),n+1)$.
\begin{center}
\begin{tikzpicture}
\matrix (m) [matrix of math nodes, row sep=2em,
column sep=2em, minimum width=2em]
{  Y_n & E(\pi_n(Y),n)  \\
  Y_{n-1} & K(\pi_n(Y),n+1) \\
};
  \begin{scope}[every node/.style={scale=.8}]
\path[->](m-1-1) edge node[above]{$r_n$}  (m-1-2);
\path[->](m-1-1) edge node[left] {$p_n$} (m-2-1);
\path[->](m-2-1) edge node[below] {$k_{n-1}$} (m-2-2);
\path[->](m-1-2) edge node[right] {$\delta$} (m-2-2);
\end{scope}
\end{tikzpicture}
\end{center}
The map $k_{n-1}$ is called a \emph{Postnikov map}. Since the Kan fibrations $\delta$ are minimal, their pullbacks $p_n$ are also minimal and all stages $Y_n$ are minimal Kan complexes.
\end{defn}

In the subsequent proposition, we compare general Postnikov towers and standard Postnikov towers with the Moore-Postnikov ones.

A morphism of Postnikow towers $f_*\colon X_*\to Y_*$ is called a \emph{weak equivalence of Postnikov towers} if all $f_n\colon X_n\to Y_n$ are \emph{simplicial weak homotopy equivalences}, i.e., they induce isomorphisms on all homotopy groups after passage to
topological spaces. Similarly, $f_*$ is an isomorphism of Postnikov towers if all $f_n$ are isomorphisms of simplicial sets.

\begin{prop}\label{TGJ1}
Let $Y$ be a simply connected minimal Kan complex. Then for any Postnikov tower $\{Y_n,p_n,\varphi_n\}$ for $Y$ there is a weak equivalence $h_*$ to the Moore-Postnikov tower $\{Y(n),p(n),\varphi(n)\}$ for $Y$ and an isomorphism $h\colon Y\to\varprojlim Y(k)$ such that the diagrams 
\begin{equation}
\begin{tikzpicture}
\matrix (m) [matrix of math nodes, row sep=2em,
column sep=2em, minimum width=2em]
{   Y & \varprojlim Y(k)=Y  \\
  Y_n & Y(n) \\
};
  \begin{scope}[every node/.style={scale=.8}]
\path[->](m-1-1) edge node[above] {$h$} (m-1-2);
\path[->](m-1-1) edge node[left] {$\varphi_n$} (m-2-1);
\path[->](m-2-1) edge node[below] {$h_n$} (m-2-2);
\path[->](m-1-2) edge node[right] {$\varphi(n)$} (m-2-2);
\end{scope}
\end{tikzpicture}\label{3E1}
\end{equation}
commute. 
If the maps $p_n$ are minimal fibrations, then both towers are isomorphic.
\end{prop}

\begin{proof}
The first part is a particular case of Theorem 5.14 in \cite{GJ}. So we have a weak equivalence $h_*$ of the Postnikov towers.
This collection of maps induces a weak equivalence $h\colon Y\to\varprojlim Y(k)=Y$ such that the diagram commutes. 

For a simply connected minimal Kan complex $Y$, the maps  $p(n)$ are minimal Kan fibrations.
If $p_n$ are also minimal fibrations then $h_n\colon Y_n\to Y(n)$ and $h\colon Y\to Y$
are weak equivalences between minimal Kan complexes, and according to Theorem 2.20 in \cite{C} they are isomorphisms.
\end{proof}

\algline
\medskip

\subsection*{Algorithmic construction of Postnikov tower}
The algorithmic approach constructs a standard Postnikov tower by specifying instructions for maps $k_{n-1}\colon Y_{n-1}\to K(\pi_n(Y),n+1)$ and $\varphi_n\colon Y\to Y_n$. Here we summarize the algorithm from \cite[Section 4]{poly} emphasizing
some facts needed for our purposes.

Consider a simplicial set $Y$ which is equipped with  effective homology $C_*(Y
)\Longleftrightarrow \EC_*(Y)$ and suppose that we have constructed its Postnikov tower up to the stage $n-1$. This means that we have $Y_{n-1}$ with  effective homology $C_*(Y_{n-1})\Longleftrightarrow \EC_*(Y_{n-1})$ and that a simplicial map $\varphi_{n-1}: Y\to Y_{n-1}$  has been computed. 
The Eilenberg-Zilber theorem implies the existence of a strong equivalence $M_*:=\operatorname{cone}_*(\varphi_{n-1})_*\Longleftrightarrow C_*(\operatorname{cyl}(\varphi_{n-1}),Y)$. Here $\operatorname{cone}_*(\varphi_{n-1})_*$ is the cone of the chain homomorphism ${\varphi_{n-1}}_*$ and $\operatorname{cyl}(\varphi_{n-1})$ is the cylinder of the simplicial map $\varphi_{n-1}$. 

Importantly, some homotopy groups of $Y$ can be identified with certain homology groups of $M_*$.
This relationship can  be derived from the composition of the Eilenberg-Zilber map $EZ$, the Hurewicz isomorphism $h$ and the connecting isomorphism $c$ of the $n$-connected pair $(\operatorname{cyl}(\varphi_{n-1}),Y)$. For $i\leq n$ we get
\begin{equation}
H_{i+1}(M_*)\xrightarrow{EZ_{i+1}} H_{i+1}(\operatorname{cyl}(\varphi_{n-1}),Y)\xrightarrow{h_{i+1}^{-1}} \pi_{i+1}(\operatorname{cyl}(\varphi_{n-1}),Y)\xrightarrow{c_{i+1}}\pi_i(Y).\label{3e4}
\end{equation}
The chain complex $M_*=\operatorname{cone}(\varphi_{n-1})_*$ 
has effective homology (see Proposition \ref{ConeEf}) since
\begin{equation}
M_*=\operatorname{cone}(\varphi_{n-1})_*=C_{*-1}(Y)\oplus C_*(Y_{n-1})\Longleftrightarrow\EC_{*-1}(Y)\oplus \EC_*(Y_{n-1})
=:\EM_*.\label{1e2}
\end{equation}
For algorithm purposes, we define in accordance with \eqref{3e4}
\begin{equation*}
\pi_n^Y:=H_{n+1}(\EM_*)\cong H_{n+1}(M_*).\label{3e3}
\end{equation*}

Now we introduce a notation. Denote by $\iota\colon C_*(Y_{n-1})\hookrightarrow M_*$ and
$\eiota\colon \EC_*(Y_{n-1})\hookrightarrow \EM_*$ the inclusions from \eqref{1e2}. The strong equivalence
$C_*(Y_{n-1})\Longleftrightarrow \EC_*(Y_{n-1})$ is a part of the strong equivalence $M_*\Longleftrightarrow  \EM_*$.
Denote by $\overline\xi_*\colon M_*\to \EM_*$ and $\xi_*\colon C_*(Y_{n-1})\to \EC_*(Y_{n-1})$  homomorphisms corresponding
to these strong equivalences. They make the following diagram commutative:
\begin{equation}
\begin{tikzpicture}
\matrix (m) [matrix of math nodes, row sep=2em,
column sep=2em, minimum width=2em]
{ \EC_*({Y_{n-1}})  & \EM_*  \\
  C_*(Y_{n-1}) & M_* \\
};
  \begin{scope}[every node/.style={scale=.8}]
\path[<-](m-1-2) edge node[right] {$\overline{\xi}_*$} (m-2-2);
\path[->](m-1-1) edge node[above] {$\iota^{\text{ef}}$} (m-1-2);
\path[<-](m-1-1) edge node[left] {$\xi_*$} (m-2-1);
\path[->](m-2-1) edge node[above] {$\iota$} (m-2-2);
\end{scope}
\end{tikzpicture}\label{3e5}
\end{equation}

In \cite{poly} it was shown that the projection of cycles $Z_{n+1}(\EM_*)$ onto $H_{n+1}(\EM_*)$ can be algorithmically extended to chains, so we have a homomorphism $\erho\colon \EM_{n+1}\to H_{n+1}(\EM_*)=\pi_n^Y$. Moreover, we can define an analogous homomorphism $\rho\colon M_{n+1}\to H_{n+1}(M_*)$ such that the diagram
\begin{equation}
\begin{tikzpicture}
\matrix (m) [matrix of math nodes, row sep=2em,
column sep=2em, minimum width=2em]
{ \EM_{n+1}  & H_{n+1}(\EM_*)  \\
  M_{n+1} & H_{n+1}(M_*) \\
};
  \begin{scope}[every node/.style={scale=.8}]
\path[<-](m-1-2) edge node[right] {$\overline{\xi}_*$} (m-2-2);
\path[->](m-1-1) edge node[above] {$\erho$} (m-1-2);
\path[<-](m-1-1) edge node[left] {$\overline{\xi}_*$} (m-2-1);
\path[->](m-2-1) edge node[above] {$\rho$} (m-2-2);
\end{scope}
\end{tikzpicture}\label{3e6}
\end{equation}
commutes. 

Now we can define the cocycle
$$\ekappa_{n-1}:=\erho\circ\eiota\colon \EC_*(Y_{n-1})\to \pi_n^Y.$$ 
We will call it an \emph{effective Postnikov cocycle}. Its cohomology class  $[\ekappa_{n-1}]\in H^{n+1}\left(\EC_*(Y_{n-1});\pi_n^Y\right)$ will be called an \emph{effective Postnikov class}. 
Next, we define the \emph{Postnikov cocycle} $\kappa_{n-1}$ as  
\[\kappa_{n-1}:=\ekappa_{n-1}\circ\xi_*=
\xi^*(\ekappa_{n-1})\in Z^{n+1}(C_{*}(Y_{n-1}),\pi_n^Y)\]
and the \emph{Postnikov class} as its cohomology class $[\kappa_{n-1}]\in H^{n+1}(C_*(Y_{n-1});\pi_n^Y)$.
The previous definitions  can now be summarized in the commutative diagram:
\begin{equation}
\begin{tikzpicture}
\matrix (m) [matrix of math nodes, row sep=2em,
column sep=2em, minimum width=2em]
{ \EC_{n+1}({Y_{n-1}})  & \EM_{n+1} &[1.5cm]  H_{n+1}(\EM_*) & \pi_n^Y \\
  C_{n+1}(Y_{n-1}) & M_{n+1} & H_{n+1}(M_*) \\
};
  \begin{scope}[every node/.style={scale=.8}]
\path[<-](m-1-2) edge node[right] {$\overline{\xi}_{n+1}$} (m-2-2);
\path[<-](m-1-3) edge node[right] {$\overline{\xi}_*$} (m-2-3);
\path[->](m-1-1) edge node[above] {$\eiota$} (m-1-2);
\path[<-](m-1-1) edge node[left] {$\xi_{n+1}$} (m-2-1);
\path[->](m-2-1) edge node[above] {$\iota$} (m-2-2);
\path[->](m-1-2) edge node[above] {$\erho$} (m-1-3);
\path[->](m-2-2) edge node[above] {$\rho$} (m-2-3);
\path[->](m-1-3) edge node[above] {$\id$} (m-1-4);
\path[->](m-2-1) edge[bend right=50] node[above] {$\kappa_{n-1}$} (m-1-4);
\path[->](m-1-1) edge[bend left] node[below] {$\ekappa_{n-1}$} (m-1-4);
\end{scope}
\end{tikzpicture}\label{3e9}
\end{equation}

For every abelian group $\pi$, the evaluation $\operatorname{ev}\colon \operatorname{SMap}(Y,K(\pi,k))\to Z^{k}(Y,\pi)$ 
from simplicial maps into $k$-cocycles on $Y$ is the map such that for a simplicial map 
$f\colon Y\to K(\pi,k)$ the cocycle  $\ev(f)$ assigns to every $k$-simplex $\sigma\in Y_k$ the value of the cocycle $f(\sigma)\in K(\pi,k)_k=Z^k(\Delta^k,\pi)$ on the unique nondegenerate $k$-simplex of $\Delta^k$. We will write $\ev(f)(\sigma)$ for this value in $\pi$. The evaluation is a bijection and  enables us to define the simplicial map
$k_{n-1}$
\begin{equation*}
k_{n-1}=\operatorname{ev}^{-1}(\kappa_{n-1})\colon Y_{n-1}\to K(\pi_n^Y,n+1).\label{3e2}
\end{equation*}
corresponding to the Postnikov cocycle $\kappa_{n-1}$. This will be the Postnikov map for our construction.
Using it, we can construct the $n$-th Postnikov stage $Y_n$ as the pullback along this map (see the definition of the standard Postnikov tower). This pullback has effective homology. To complete the construction of the $n$-th stage we have to find a simplicial map $\varphi_n\colon Y\to Y_{n}$ which fits into  diagram \eqref{3d2} of the Postnikov tower. Since $Y_n$ is the pullback, the required map is given by the couple of maps $\varphi_{n-1}:Y\to Y_{n-1}$ and $l_n:Y\to E(\pi_n^Y,n)$.
The map $l_n$ corresponds via bijective evaluation $\operatorname{SMap}(Y, E(\pi_n^Y,n))\to
C^n(Y;\pi_n^Y)$ to the cochain $\lambda_n$ given by the composition
\begin{equation*}
C_n(Y)\hookrightarrow M_{n+1}\xrightarrow{\overline\xi} \EM_{n+1}\xrightarrow{\E{\rho}}\pi_n^Y.
\end{equation*}
In \cite{poly} it is proved that $k_{n-1}$, $Y_n$ and $\varphi_n$ satisfy all properties in the definition  of the standard Postnikov tower.

\begin{defn} The Postnikov tower constructed above will be called an \emph{effective Postnikov tower} of the simplicial set $Y$. 
\end{defn}

At the end of this section, we recall the notion of fundamental class and transgression and their relationship with  Postnikov classes. These relations play a crucial role in the proof of Theorem \ref{nec}.
Let $\iota_k:K(\pi,k)_k\to \pi$ be the mapping assigning to each $\pi$-valued cocycle $\sigma\in Z^k(\Delta^k,\pi)$ its value on
the unique $k$-simplex of $\Delta^k$, i.e. $\iota_k=\ev(\id_{K(\pi,k)})$. The cohomology class $[\iota_k]\in H^k(K(\pi,k); \pi)$ is called the fundamental class. The next lemma provides the relations between fundamental classes and the Postnikov class $[\kappa_{n-1}]$ constructed above.

\begin{lem}\label{trans}
Denote by $\tau^K$ and $\tau$ the transgressions in the Serre long exact sequences of fibrations $K(\pi_n^Y,n)\to E(\pi_n^Y,n)\to 
K(\pi_n^Y,n+1)$ and $K(\pi_n^Y,n)\to Y_n\to Y_{n-1}$, respectively. Then
$$\tau^K([\iota_n])=[\iota_{n+1}] \quad \text{and}\quad \tau([\iota_n])=k_{n-1}^*([\iota_{n+1}])=[\kappa_{n-1}].$$
\end{lem}

\begin{proof}
Using the fact that $H^k(K(\pi^Y_n,k);\pi_n^k)\cong \operatorname{Hom}(H_k(K(\pi^Y_n,k),\pi^Y_n)$ the transgression $\tau^K$ is dual to a homomorphism $\tau_K\colon H_{n+1}(K(\pi_n^Y,n+1))\to H_{n}(K(\pi_n^Y,n))$ which  can be represented by a homomorphism $\tau_K$ on the level of chain complexes: 
\begin{equation*}
\begin{tikzpicture}
\matrix (m) [matrix of math nodes, row sep=2em,
column sep=2em, minimum width=2em]
{ {} & C_{n+1}(E(\pi_n^Y,n))   & \frac{C_{n+1}(E(\pi_n^Y,n))}{C_{n+1}(K(\pi_n^Y,n)) }  & C_{n+1}(K(\pi_n^Y,n+1))  \\
  C_{n}(K(\pi_n^Y,n)) & C_{n}(E(\pi_n^Y,n)) & {} & {}\\
};
  \begin{scope}[every node/.style={scale=.8}]
\path[->>](m-1-3) edge node[above] {$\delta_*$} (m-1-4);
\path[->>](m-1-2) edge node[above] {$\operatorname{pr}_*$} (m-1-3);
\path[>->](m-2-1) edge   (m-2-2);
\path[->](m-1-2) edge node[right] {$\partial$} (m-2-2);
\path[->](m-1-4) edge[bend left] node[below] {$\tau_K$} (m-2-1);
\path[->](m-1-4) edge[bend right] node[above] {$\varphi$} (m-1-2);
\end{scope}
\end{tikzpicture}
\end{equation*}
To get a precise definition of $\tau_K$, one needs to specify the right inverse $\varphi$ of $(\delta\operatorname{pr})_*$.
Take a simplex $\sigma$ of $K(\pi_n^Y,n+1)_{n+1}$ uniquely determined by its value $\ev(\id_{K(\pi_n^Y,n+1)})(\sigma)\in\pi_n^Y$. Define $\varphi(\sigma)$ as the cochain 
in $C^n(\Delta^{n+1},\pi_n^Y)$ (i.e. an element of $C_{n+1}(E(\pi_n^Y,n)$) satisfying $\varphi(\sigma)(d^0)=\ev(\id_{K(\pi_n^Y,n+1)})(\sigma)$ and $\varphi(\sigma)(d^i)=0$ for $i\ge 1$ where $d^i\colon \Delta^{n}\to\Delta^{n+1}$ are the standard faces in $\Delta^{n+1}$.  
Then $\partial(\varphi(\sigma))$ is a simplex in $K(\pi_n^Y,n)_n$
with the property that
$$\ev(\id_{K(\pi_n^Y,n)})(\partial(\varphi(\sigma)))=\ev(\id_{K(\pi_n^Y,n+1)})(\sigma).$$
Hence $\tau^K([\iota_n])=[\iota_{n+1}]$.

To prove the latter statement, consider the pullback diagram of fibrations: 
\begin{equation*}
\begin{tikzpicture}
\matrix (m) [matrix of math nodes, row sep=2em,
column sep=2em, minimum width=2em]
{ K(\pi_n^Y,n)  & K(\pi_n^Y,n)   \\
  Y_n & E(\pi_n^Y,n) \\
  Y_{n-1} & K(\pi_n^Y,n+1) \\
};
  \begin{scope}[every node/.style={scale=.8}]
\path[-](m-1-1) edge node[above] {$\id$} (m-1-2);
\path[->](m-1-1) edge  (m-2-1);
\path[->](m-2-1) edge  (m-2-2);
\path[->](m-1-2) edge  (m-2-2);
\path[->](m-2-1) edge  (m-3-1);
\path[->](m-2-2) edge node[right] {$\delta$} (m-3-2);
\path[->](m-3-1) edge node[above] {$k_{n-1}$} (m-3-2);
\end{scope}
\end{tikzpicture}
\end{equation*}
Since the transgression is natural,  we get
\begin{align*}\tau([\iota_n])&=k_{n-1}^*(\tau^K([\iota_n]))=k_{n-1}^*([\iota_{n+1}])\\
&=k_{n-1}^*([\ev(\id_{K(\pi_n^Y,n+1)})])=
[\ev(\id_{K(\pi_n^Y,n+1)}\circ k_{n-1})]\\
&=[\ev(k_{n-1})]=[\kappa_{n-1}].
\end{align*}

\end{proof}

\section{Necessary and sufficient conditions}

This section aims to establish necessary and sufficient conditions in terms of their effective Postnikov towers
to decide if the geometric realizations of two finite simply connected simplicial sets $X$ and $Y$ are homotopy equivalent. First, we move from the geometric realizations of simplicial sets to their substitutions as minimal Kan complexes.

\begin{lem}\label{BL}
Let $X$ and $Y$ be simply connected simplicial sets. Let $\{X_{n}, p_n^X,\varphi_n^X\}$ and $\{Y_{n},p_n^Y,\varphi_n^Y \}$ be their standard Postnikov towers. Put $X':=\varprojlim X_n$ and $Y':=\varprojlim Y_n$.  Then there is an isomorphism 
$$[|X|,|Y|]\cong [X',Y'].$$

In particular, $|X|$ and $|Y|$ are homotopy equivalent if and only if there is a simplicial weak homotopy equivalence $f\colon X'\to Y'$.
\end{lem}

\begin{proof} The desired isomorphism is a composition of two bijections. First, using the induced maps
$\varphi^X\colon X\to X'$ and $\varphi^Y\colon Y\to Y'$ one can define the isomorphism
$[|X|,|Y|]\cong [|X'|,|Y'|]$. Next, we use a canonical isomorphism $[A,B]\cong [|A|,|B|]$
for Kan complexes $A$ and $B$ applied to $A=X'$ and $B=Y'$.
\end{proof}

It is natural to seek an inductive criterion that finds weak equivalences between the stages of the Postnikov towers of simplicial sets $X$ and $Y$ successively from lower to higher stages. Once such a criterion is available, one needs to relate it with the existence of homotopy equivalence between the geometric realizations of $X$ and $Y$. The following theorem gives a necessary condition. 
In the category of topological spaces, this condition was derived and used by Kahn in  \cite{kahni} and \cite{kahns}. However, our algorithmic construction of the Postnikov tower of simplicial sets differs from Kahn's, and there is no immediate way to restate his result from topological spaces to simplicial sets. In the simplicial context, the key construction ingredient is a minimal model of Eilenberg-MacLane space. 

\begin{thm}\label{nec}
Let $X$ and $Y$ be simply connected simplicial sets. Let $\{X_{n}, p_n^X,\varphi_n^X\}$ and $\{Y_{n},
p_n^Y,\varphi_n^Y \}$ be their effective Postnikov towers with Postnikov cocycles $\kappa_n^X$, $\kappa_n^Y$,
respectively. Put $X':=\varprojlim X_n$ and  $Y':=\varprojlim Y_n$, and let $\varphi^{X'}_n\colon X'\to X_n$ and $\varphi^{Y'}_n\colon Y'\to Y_n$ be canonical maps. If there is a simplicial map $f\colon X'\to Y'$, then there are maps $f_n\colon X_n\to Y_n$ such that all diagrams
\begin{equation*}
\begin{tikzpicture}
\matrix (m) [matrix of math nodes, row sep=2em,
column sep=2em, minimum width=2em]
{  X_n & Y_n  \\
  X_{n-1} & Y_{n-1} \\};
  \begin{scope}[every node/.style={scale=.8}]
\path[->](m-1-1) edge node[above] {$f_n$} (m-1-2);
\path[->>](m-1-2) edge node[right] {$p_n^Y$} (m-2-2);
\path[->>](m-1-1) edge node[left] {$p_n^X$} (m-2-1);
\path[->](m-2-1) edge node[below] {$f_{n-1}$} (m-2-2);
\end{scope}
\end{tikzpicture}
\hskip 2cm
\begin{tikzpicture}
\matrix (m) [matrix of math nodes, row sep=2em,
column sep=2em, minimum width=2em]
{  X' & Y'  \\
  X_{n} & Y_{n} \\};
  \begin{scope}[every node/.style={scale=.8}]
\path[->](m-1-1) edge node[above] {$f$} (m-1-2);
\path[->](m-1-2) edge node[right] {$\varphi^{Y'}_n$}  (m-2-2);
\path[->](m-1-1) edge node[left] {$\varphi^{X'}_n$} (m-2-1);
\path[->](m-2-1) edge node[below] {$f_{n}$} (m-2-2);
\end{scope}
\end{tikzpicture}
\end{equation*}
commute strictly. Moreover, the  Postnikov classes $[\kappa_{n-1}^X]$ and $[\kappa_{n-1}^Y]$ are in the relation
\begin{equation}
\gamma_*([\kappa_{n-1}^X])=f_{n-1}^*([\kappa_{n-1}^Y])\label{eq}
\end{equation}
where $\gamma_*\colon H^{n+1}(X_{n-1};\pi_n^X)\to H^{n+1}(X_{n-1};\pi_n^Y)$ is induced by a homomorphism
$\gamma\colon \pi_n^X\to \pi_n^Y$ which captures the map $f$ in a suitable manner.

If $f$ is a simplicial weak homotopy equivalence, then all $f_n$ are simplicial weak homotopy equivalences (and even isomorphisms of simplicial sets), and $\gamma$ 
is a group isomorphism
\end{thm}

\begin{proof}
Start with deriving both diagrams. Note that $\{X_n,\kappa_n^X\}$ and $\{Y_n,\kappa_n^Y\}$ with canonical maps $\varphi_n^{X'}\colon X'\to X_n$ and $\varphi_n^{Y'}\colon Y'\to Y_n$ are standard Postnikov towers for $X'$ and $Y'$, respectively. Apply  Proposition \ref{TGJ1} for these Postnikov towers to compare them with the Moore-Postnikov towers
of $X'$ and $Y'$, respectively. The isomorphisms $h\colon Y'\to Y'$, $g\colon X'\to X'$ between the towers together with the functoriality of the Moore-Postnikov towers provide:
\begin{center}
\begin{tikzpicture}
\matrix (m) [matrix of math nodes, row sep=2em,
column sep=4em, minimum width=2em]
{  X' & X' & Y' & Y'  \\
  X_{n} & X'(n) & Y'(n) & Y_{n} \\
  X_{n-1} & X'(n-1) & Y'(n-1) & Y_{n-1}\\};
  \begin{scope}[every node/.style={scale=.8}]
\path[->](m-1-1) edge node[above] {$g$} node[below] {$\cong$} (m-1-2);
\path[->](m-2-1) edge node[below] {$\cong$} (m-2-2);
\path[->](m-3-1) edge node[below] {$\cong$} (m-3-2);
\path[->](m-1-2) edge node[above] {$hfg^{-1}$} (m-1-3);
\path[->](m-2-2) edge node[below] {$(hfg^{-1})(n)$} (m-2-3);
\path[->](m-3-2) edge node[below] {$(hfg^{-1})(n-1)$} (m-3-3);
\path[<-](m-1-3) edge node[above] {$h$} node[below] {$\cong$} (m-1-4);
\path[<-](m-2-3) edge node[below] {$\cong$} (m-2-4);
\path[<-](m-3-3) edge node[below] {$\cong$} (m-3-4);
\path[->](m-1-2) edge node[left] {$\varphi^{X'}({n})$} (m-2-2);
\path[->](m-1-1) edge node[left] {$\varphi^{X'}_{n}$} (m-2-1);
\path[->](m-2-1) edge node[left] {$p_{n}^X$} (m-3-1);
\path[->](m-2-2) edge node[left] {$p^X({n})$} (m-3-2);
\path[->](m-2-3) edge node[right] {$p^Y({n})$} (m-3-3);
\path[->](m-2-4) edge node[right] {$p_{n}^Y$} (m-3-4);
\path[->](m-1-3) edge node[right] {$\varphi^{Y'}({n})$} (m-2-3);
\path[->](m-1-4) edge  node[right] {$\varphi^{Y'}_n$} (m-2-4);
\path[->](m-3-1) edge[bend right] node[below] {$f_{n-1}$} (m-3-4);
\path[->](m-1-1) edge[bend left] node[above] {$f$} (m-1-4);
\end{scope}
\end{tikzpicture}
\label{f3}
\end{center}
The commutative diagram defines horizontal maps $f_n\colon X_n\to Y_n$ for all $n\in \mathbb N$ and gives the commutative diagrams
in our theorem. If $f$ is a simplicial weak homotopy equivalence then maps $(hfg^{-1})(n)$ and $f_n$ are also simplicial weak homotopy equivalences between minimal Kan complexes, so they are isomorphisms by Theorem 2.20 in \cite{C}.

It remains to show that equation \eqref{eq} holds. 
To complete this task, take the commutative diagram:
\begin{equation*}
\begin{tikzpicture}
\matrix (m) [matrix of math nodes, row sep=2em,
column sep=2em, minimum width=2em]
{ K(\pi_n^X,n)  & K(\pi_n^Y,n)   \\
 X_n & Y_n \\
  X_{n-1} & Y_{n-1} \\
};
  \begin{scope}[every node/.style={scale=.8}]
\path[->](m-1-1) edge node[above] {$\bar f$} (m-1-2);
\path[->](m-1-1) edge  (m-2-1);
\path[->](m-2-1) edge node[above] {$f_{n}$}  (m-2-2);
\path[->](m-1-2) edge  (m-2-2);
\path[->](m-2-1) edge node[left] {$p_n^X$} (m-3-1);
\path[->](m-2-2) edge node[right] {$p_n^Y$} (m-3-2);
\path[->](m-3-1) edge node[below] {$f_{n-1}$} (m-3-2);
\end{scope}
\end{tikzpicture}
\end{equation*}
Consider the homorphism $\gamma\colon\pi_n^X\to\pi_n^Y$ determined by the commutative square:
\begin{equation*}
\begin{tikzpicture}
\matrix (m) [matrix of math nodes, row sep=2em,
column sep=2em, minimum width=2em]
{ H_n(K(\pi_n^X,n))  &[2cm] H_n(K(\pi_n^Y,n))   \\
\pi_n^X  & \pi_n^Y\\
};
  \begin{scope}[every node/.style={scale=.8}]
\path[->](m-1-1) edge node[above] {$\bar f_*$} (m-1-2);
\path[->](m-1-1) edge node[left] {$\ev(\id_{K(\pi_n^X,n)}),\cong$}  (m-2-1);
\path[->](m-2-1) edge node[below] {$\gamma$}  (m-2-2);
\path[->](m-1-2) edge node[right] {$\ev(\id_{K(\pi_n^Y,n)}),\cong$}  (m-2-2);
\end{scope}
\end{tikzpicture}
\end{equation*}

Using the naturality of the transgression and Lemma \ref{trans} we obtain: 
\begin{equation*}
 f_{n-1}^*([\kappa^Y_{n-1}])=f_{n-1}^*(\tau([\iota_n^Y]))=\tau(\bar f^*([\iota_n^Y]))=\tau(\gamma_*([\iota_n^X])
=\gamma_*(\tau([\iota_n^X]))=\gamma_*([\kappa_{n-1}^X]).
\end{equation*}
Finally, if $f_n$ and $f_{n-1}$ are simplicial weak homotopy equivalences, then $\bar f$ is also a simplicial weak homotopy equivalence.
Next, the induced map $\bar f_*$ is an isomorphism (see the last diagram) which implies that $\gamma$ is an isomorphism. 
\end{proof}

\algline
\medskip

Consequently, we can formulate a well-known assertion that the homotopy type of a finite dimensional simplicial complex is given
only by a finite part of its Postnikov tower. From an algorithmic point of view, this fact plays a crucial role.

\begin{cor}\label{criterion}
Let $X$ and $Y$ be finite simply connected simplicial sets of dimensions $\le d$ with effective Postnikov towers $\{X_n\}$ and $\{Y_n\}$, respectively. Then $|X|$ and $|Y|$ are homotopy equivalent if and only if there is a simplicial weak homotopy equivalence
\[f_d\colon X_d\to Y_d.\]
\end{cor}

\begin{proof}
If $|X|$ and $|Y|$ are homotopy equivalent, then by Lemma \ref{BL} there is a simplicial weak homotopy equivalence 
$f\colon X':=\varprojlim X_n\to Y':=\varprojlim Y_n$ and by Theorem \ref{nec}
there is a simplicial weak homotopy equivalence $f_d\colon X_d\to Y_d$.

Conversely, suppose that there is a simplicial weak homotopy equivalence $f_d\colon X_d\to Y_d$. Since $|X|$ is a CW-complex of dimension $d$ and $|\varphi_d^Y|:|Y|\to |Y_d|$ is an isomorphism in homotopy groups $\pi_i$ for $i\le d$ and an epimorphism for $i=d+1$, the induced map $|\varphi_d^Y|_*\colon [|X|,|Y|]\cong [|X|,|Y_d|]$ is a bijection.
Let $F\colon |X|\to |Y|$ be a map which homotopy class corresponds to the homotopy class of $|f_d\circ\varphi_d^X|\colon
|X|\to |Y_d|$. Then $F$ induces isomorphisms in homotopy groups $\pi_i$ for $i\le d$. Since $|X|$ and $|Y|$ have dimensions $\le d$, it induces  also isomorphisms in homology groups $H_i(-;\mathbb Z)$ for all $i$. Therefore $F\colon |X|\to|Y|$ is a homotopy equivalence by Whitehead Theorem. 
\end{proof}

The following theorem and its proof show that equation \eqref{eq}, where $\gamma$ is an isomorphism, is sufficient for algorithmic construction of a simplicial map lifting a given homotopy equivalence between Postnikov $(n-1)$-stages to a homotopy equivalence between $n$-stages.

\begin{thm}\label{suf}
Let $X$ and $Y$ be simply connected simplicial sets with effective homology. Let $\{X_{n}, p_n^X, \varphi_n^X\}$ and 
$\{Y_{n}, p_n^Y,\varphi_n^Y\}$ be their effective Postnikov towers with Postnikov cocycles $\kappa_{n}^X$ and $\kappa_{n}^Y$, respectively. Assume that there are a computable simplicial map $f_{n-1}\colon X_{n-1}\to Y_{n-1}$ and an isomorphism $\gamma\colon \pi_{n}^X\to\pi_{n}^Y$ such that relation \eqref{eq}
\begin{equation*}
\gamma_*[{\kappa}_{(n-1)*}^{X}]=f_{n-1}^*[\kappa_{(n-1)*}^Y]\label{e3}
\end{equation*}
holds. Then:
\begin{itemize}
\item[(1)] A lift  of $f_{n-1}p_n^X$ to a map $f_n\colon X_n\to Y_n$ exists, i.e. the diagram 
\begin{center}
\begin{tikzpicture}
\matrix (m) [matrix of math nodes, row sep=2em,
column sep=2em, minimum width=2em]
{  X_n & Y_n  \\
  X_{n-1} & Y_{n-1} \\};
  \begin{scope}[every node/.style={scale=.8}]
\path[->](m-1-1) edge node[above] {$f_n$} (m-1-2);
\path[->>](m-1-2) edge node[right] {$p_n^Y$} (m-2-2);
\path[->>](m-1-1) edge node[left] {$p_n^X$} (m-2-1);
\path[->](m-2-1) edge node[below] {$f_{n-1}$} (m-2-2);
\end{scope}
\end{tikzpicture}
\end{center}
commutes.
\item[(2)]If we consider $X_n$ and $Y_n$ as subsets of $X_{n-1}\times E(\pi_n^X,n)$
and $Y_{n-1}\times E(\pi_n^Y,n)$, respectively, such a map $f_n$ can be defined by 
\begin{equation}
f_{n}(x,y)=(f_{n-1}(x),\gamma_*(y + \omega(x)))\label{form}
\end{equation}
for some/any $\omega\colon X_{n-1}\to E(\pi_{n}^X,n)$ satisfying
\begin{equation}
\ev^{-1}(\gamma_*^{-1}f^*_{n-1}\kappa^Y_{n-1}-\kappa^X_{n-1})=\delta\omega.\label{cond}
\end{equation}
\item[(3)]There is an algorithm which computes one such $\omega$.
\item[(4)]The ambiguity in the definition of $\omega$ is up to maps $c\colon X_{n-1}\to K(\pi_{n}^X,n)$. 
\end{itemize}

If $f_{n-1}$ is a simplicial weak homotopy equivalence, so is $f_n$ (and even an isomorphism of simplicial sets). 
\end{thm}

\begin{app}\label{ap}
Under the same assumptions consider the natural fibrewise action $+\colon Y_{n}\times K(\pi_{n}^Y,n)\to Y_{n}$. Then
any two maps $f_n,f_n'\colon X_n\to Y_n$ making the diagram in Theorem \ref{suf} commutative up to homotopy differ by a map $\bar d\colon X_{n}\to K(\pi_{n}^Y,n)$ such that 
\begin{equation*}
f_{n}' \sim f_{n}+ \bar d.
\end{equation*}
Furthermore, if we assume that for the fibre inclusions  $\iota^X\colon K(\pi_{n}^X,n)\to X_{n}$ and $\iota^Y\colon K(\pi_{n}^Y,n)\to Y_{n}$
\[f_{n}\circ \iota^X\sim f_{n}'\circ \iota^X\colon K(\pi_{n}^X,n)\to \iota^Y(K(\pi_{n}^Y,n)),\] 
then there exists a map $c\colon X_{n-1}\to K(\pi_{n}^X,n)$ satisfying
\begin{align*}
f_{n}'&\sim f_{n}+\gamma_* \circ c\circ p_{n}^X
\end{align*}
where $\gamma_*\colon K(\pi_{n}^X,n)\to K(\pi_{n}^Y,n)$ is determined by the isomorphism $\gamma$ from
(\ref{eq}).
\end{app}

\medskip
\begin{proof}[Proof of Theorem \ref{suf}]
Consider the following diagram 
\begin{equation}
\begin{tikzcd}[
  row sep=scriptsize,
  column sep=scriptsize,
]
{X_{n}}\ar[rr,dashed," "]\ar[dr]\ar[dd,->>,"{p_n^X}"] & & Y_{n}\ar[dr]\ar[dd,->>,"p_n^Y",near end] & \\
& E(\pi^X_{n},n)\ar[rr,dashed,"\gamma_*",near start,crossing over]\ar[dd,->>,"\delta",crossing over,near start,swap] & & E(\pi^Y_n,n)\ar[dd,->>,"\delta",swap] & \\
X_{n-1}\ar[rr,"f_{n-1}",near start]\ar[dr,"{k_{n-1}^X}"] & & Y_{n-1}\ar[dr,"k_{n-1}^Y"] & \\
& K(\pi^X_{n},n+1)\ar[rr,"\gamma_*"] & & K(\pi^Y_n,n+1) &
\end{tikzcd}\label{cube}
\end{equation}
where the squares given by solid arrows commute. Complete the diagram by the map
$$\gamma_*\colon E(\pi_n^X, n)\to E(\pi_n^Y, n)$$
which is induced by the isomorphism $\gamma\colon \pi_n^X\to\pi_n^Y$. Then the front square is also commutative. Our aim is to give an algorithmic construction of a map
\begin{equation*}
f_n\colon X_{n-1}\to Y_{n-1}\times E(\pi_n^Y, n)
\end{equation*}
such that its image lies in $Y_n\subseteq Y_{n-1}\times E(\pi_n^Y, n)$ and it makes the remaining squares commutative.

Since $\gamma_*[\kappa_{n-1}^X]=f^*_{n-1}[\kappa_{n-1}^Y]$, the cohomology class of the difference of corresponding maps 
$$k^Y_{n-1}\circ f_{n-1}-\gamma_*\circ k_{n-1}^X\colon X_{n-1}\to K(\pi_n^Y,n+1)$$
is trivial. This applies also to the map
$\gamma_*^{-1}\circ k^Y_{n-1}\circ f_{n-1}-k_{n-1}^X\colon X_{n-1}\to K(\pi_n^X,n+1)$. So there is a lift 
$\omega\colon X_{n-1}\to E(\pi_n^X,n)$ in the diagram
\begin{center}
\begin{tikzpicture}
\matrix (m) [matrix of math nodes, row sep=3.0em,
column sep=9em, minimum width=2em]
{   & E(\pi_n^X,n)  \\
  X_{n-1} & K(\pi_n^X,n+1) \\
};
  \begin{scope}[every node/.style={scale=.8}]
\path[->](m-2-1) edge node[below] {$\gamma_*^{-1}\circ k_{n-1}^Y\circ f_{n-1}-k_{n-1}^X$} (m-2-2);
\path[->>](m-1-2) edge node[right] {$\delta$} (m-2-2);
\path[->](m-2-1) edge node[auto] {$\omega$} (m-1-2);
\end{scope}
\end{tikzpicture}
\end{center}
Condition (\ref{cond}) only rewrites the commutativity of this diagram in terms of Postnikov cocycles and the evaluation map. 

According to Lemma 2.11 in \cite{aslep}  one can compute such a single lift algorithmically. 
Any two lifts differ by a map $c\colon X_{n-1}\to K(\pi_n^X,n)$ to the fiber of the fibration $\delta$. 

For any lift $\omega\colon X_{n-1}\to E(\pi_n^X, n)$ let us define $f_n\colon X_{n-1}\times E(\pi_n^X,n)\to Y_{n-1}\times E(\pi_n^Y,n)$  
by the formula
\[f_n(x,y)=\left(f_{n-1}(x),\gamma_*(y+\omega(x))\right).\]
The restriction of this map to $X_n$ has image in $Y_n$ since for $(x,y)\in X_n\subseteq X_{n-1}\times E(\pi_n^X,n)$ we have $\delta y=k_{n-1}^X(x)$ and that is why 
\begin{align*}
\delta\left(\gamma_*(y+\omega(x))\right)&=\gamma_*\delta(y)+\gamma_*\delta(\omega(x))=
\gamma_*k_{n-1}^X(x)+\gamma_*(\gamma_*^{-1}{k_{n-1}^Y}f_{n-1}-k_{n-1}^X)(x)\\
&=k_{n-1}^Yf_{n-1}(x).
\end{align*}
It also shows that $f_n$ is a lift of $f_{n-1}$. 
\end{proof}

\begin{proof}[Proof of Addendum \ref{ap}]
The assumption says that 
\[p^Y_n\circ f_n\sim f_{n-1}\circ p^X_{n}\sim p^Y_n\circ f'_n. \]
Since $p_n^Y$ is a principal $K(\pi_n^Y,n)$-fibration, there is a map $\bar d\colon X_n\to K(\pi_n^Y,n)$ such that $f'_n\sim f_n+\bar d$.
Consider the long Serre exact sequence of the fibration $p_n^X$:
\[ \dots \to  H^n(X_{n-1},\pi_n^Y)\to H^n(X_{n},\pi_n^Y)\to H^n(K(\pi_n^X,n),\pi_n^Y)\to \dots\]
If the maps $f_n$ and $f_n'$ induce homotopic maps on fibres, the image of the class $[\bar d]=[f'_n-f_n]\in H^n(X_{n},\pi_n^Y)$ vanishes in $H^n(K(\pi_n^X,n);\pi_n^Y)$, so a difference of $f_n$ and $f'_n$ is homotopic to a certain map $d\colon X_{n-1}\to K(\pi_n^Y,n)$, i.e. $f'_n\sim f_n+d\circ p_n^X$. As the isomorphism $\gamma\colon\pi_n^X\to \pi_n^Y$ induces simplicial sets isomorphism 
$\gamma_*\colon K(\pi_n^X,n)\to K(\pi_n^Y,n)$, there is a map $c\colon X_{n-1}\to K(\pi_n^X,n)$ such that $f'_n\sim f_n+\gamma_*\circ c\circ p_n^X$.
\end{proof}

\begin{rem}
Here, we want to point out that the specific form of tower maps \eqref{form} is fully matching with the characterisation of simplicial maps between principal twisted cartesian products (PTCP) from May \cite[Section 31]{M}. For our specific PTCP's $X_n=E(\tau^X)=K(\pi_n^X,n)\times_{\tau^X} X_{n-1}$ and $Y_n=E(\tau^Y)=K(\pi^Y_n,n)\times_{\tau^Y} Y_{n-1} $ with $\tau^X=\tau\circ k^X_{n-1}$ and $\tau^Y=\tau\circ k^Y_{n-1}$, respectively, where $\tau$ is the twisted operator for the twisted product 
$E(\pi_n,n)=K(\pi_n,n)\times_\tau K(\pi_n,n+1)$. According to \cite{M} the morphisms of PTCP's $X_n=E(\tau^X)\to E(\tau^Y)=Y_n$ have the form
\[\theta(y,x)=(\alpha(y)+\xi(x),\beta(x)),\]
where $y\in K(\pi_n^X,n)$, $x\in X_{n-1}$, $\beta\colon X_{n-1}\to Y_{n-1}$ is a simplicial map, $\alpha\colon K(\pi^X_n,n)\to K(\pi^Y_n,n)$ is a simplicial homomorphism and $\xi\colon X_{n-1}\to K(\pi^Y_n,n)$ is a function satisfying the certain simplicial identities denoted as (U) in \cite{M}. The transition between PTCP's and pullbacks provides the simplicial isomorphism $\varphi\colon X_{n-1}\times_{k^X_n} E(\pi_n^X,n)\to E(\tau^X)$
\begin{equation}\label{ptcp}
\varphi(x,y) = (y - \psi(k^X_n(x)), x)
\end{equation}
where $\psi\colon K(\pi_n^X,n+1)\to K(\pi_n^X,n)$ is the pseudo-cross section. Analogically, identify the stage $Y_n$ with its PTCP's counterpart $E(\tau^Y)$ via $\overline{\varphi}\colon  E(\tau^Y)\to Y_{n-1}\times_{k^Y_n} E(\pi_n^Y,n)$
\[\overline{\varphi}(y,x) = (x, y +\psi(k^Y_n(x))).\]
Thus, the form of morphism (\ref{ptcp}) translates to $\overline{\varphi}\theta\varphi\colon X_{n-1}\times_{k^X_n} E(\pi_n^X,n)\to Y_{n-1}\times_{k^Y_n} E(\pi_n^Y,n)$ 
\[(x,y)\longmapsto (\beta(x), \alpha(y)+ \psi(k^Y_n(\beta(x)))-\alpha(\psi(k^X_n(x)))+\xi(x)).\]
One can easily examine from underlying definitions that the map \[\Omega(x)=\psi(k^Y_n(\beta(x)))-\alpha(\psi(k^X_n(x)))+\xi(x)\] 
represents an arbitrary simplicial map $X_{n-1}\to E(\pi_n^Y,n)$. This description clearly matches with our form \eqref{form}.
\end{rem}

\section{Effective homotopy equivalence and self-equivalence}

To decide if two spaces are homotopy equivalent, it is essential to describe the homotopy equivalences which can be constructed inductively by the algorithm from Theorem \ref{suf}.
It turns out that such homotopy equivalences behave well concerning the composition of maps and represent all homotopy equivalences up to homotopy. Simultaneously, it will be helpful to give a similar description of homotopy self-equivalence, described in Corollary \ref{cor}. 

This corollary is a crucial step in proving Theorem A since it translates our problem of algorithmically deciding whether two finite simplicial sets are homotopy equivalent to a problem in algorithmic group theory.

The first part of the section introduces effective versions of homotopy equivalences between the stages of the Postnikov towers. The definition is motivated by Theorem \ref{suf}.

\begin{defn}
Let $\{X_{n}\}$ and $\{Y_{n}\}$ be effective Postnikov towers for simply-connected simplicial sets $X$ and $Y$, respectively. Denote  by
$\iso(X_{n},Y_{n})$ the set of all simplicial weak homotopy equivalences $X_{n}\to Y_{n}$ and $\aut(X_{n})$ the set of all simplicial weak homotopy self-equivalences $X_{n}\to X_{n}$. Let $\Iso(X_{n},Y_{n})$
and $\Aut(X_{n})$ be the sets of their homotopy classes, respectively, i.e.
\[\Iso(X_{n},Y_{n})=\iso(X_{n}, Y_{n})/\sim,\quad \Aut(X_{n})=\aut(X_{n})/\sim.\] 
The sets of effective homotopy equivalences between the Postnikov stages $X_{n}$ and $Y_{n}$ are defined inductively as
\begin{align*}
\isoef(X_{1},Y_{1}):=& \{\id:X_{1}=*\to *=Y_{1}\},\\
\isoef(X_{n},Y_{n}):=&\{f_{n}\colon X_{n}\to Y_{n}|\ f_{n} \text{ has the form } \eqref{form} \text{ with } f_{n-1}\in \isoef(X_{n-1},Y_{n-1}) \\
               &\ \text{ and satisfies the condition } \eqref{cond}\},\ n\ge 2.
\end{align*}
The symbol $\Isoef(X_{n},Y_{n})$ will stand for the set of their homotopy classes
\[\Isoef(X_{n},Y_{n})=\isoef(X_{n},Y_{n})/\sim.\] 
Furthermore, we define the set of effective homotopy self-equivalences $X_{n}\to X_{n}$ as
\[\autef(X_{n}):=\isoef(X_{n},X_{n}),\quad \Autef(X_{n}):=\Isoef(X_{n},X_{n})\]
\end{defn}

Notice that using this new notation Theorem \ref{suf} says that if  $f_{n-1}\in \isoef(X_{n-1},Y_{n-1})$
and an isomorphism $\gamma:\pi_n^X\to\pi_n^Y$ satisfy condition (\ref{eq}) then there is an algorithm
constructing $f_n\in\isoef(X_{n},Y_{n})$ which is a lift of $f_{n-1}$. At the same time, the theorem says that $f_n$ is a simplicial weak homotopy equivalence as $f_1=\id$ has the same property.

The following statement shows that the sets of effective homotopy equivalences have the nice properties promised above.

\begin{prop}\label{isoef}
Let $X$, $Y$ and $Z$ be simply connected simplicial sets with effective homology. Let $\{X_{n}, k_{n}^X, \varphi_{n}^X\}$, 
$\{Y_{n}, k_{n}^Y,\varphi_{n}^Y\}$ and $\{Z_{n}, p_{n}^Z,\varphi_{n}^Z\}$  be their effective Postnikov towers. 
\begin{enumerate}
    \item If $f\in\isoef(X_{n},Y_{n})$ and $g\in\isoef(Y_{n},Z_{n})$, then $g\circ f\in \isoef(X_{n},Z_{n})$.
    \item If $f\in\isoef(X_{n},Y_{n})$, then the inverse $f^{-1}\in\isoef(Y_{n},X_{n})$.
    \end{enumerate}
Particularly, the sets $\autef(X_{n})$ together with composition of maps are groups.    
\end{prop}
\smallskip

\begin{proof}
 Consider the identifications $X_{n}\subseteq X_{n-1}\times E(\pi_{n}^X,n)$, $Y_{n}\subseteq Y_{n-1}\times E(\pi_{n}^Y,n)$ and $Z_{n}\subseteq Z_{n-1}\times E(\pi_{n}^Z,n)$ and write $f$ and $g$ explicitly as:
\begin{align*}
    f(x,y)&=(f_{n-1}(x),\gamma_*(y+\omega(x))),\\
     g(x',y')&=(g_{n-1}(x'),\gamma'_*(y'+\omega'(x'))),
\end{align*}
The composition $g\circ f$ is 
\begin{align*}
g\circ f(x,y)&=((g_{n-1}\circ f_{n-1})(x),(\gamma'\gamma)_*(y+\omega(x))+\gamma'_*(\omega'(f_{n-1}(x))))\\
&=\Big[g_{n-1}\circ f_{n-1}(x),(\gamma'\gamma)_*\Big(y+\big(\omega(x)+\gamma_*^{-1}\omega'f_{n-1}(x)\big)\Big)\Big]
\end{align*}
which has the form \eqref{form}:
\[g\circ f(x,y)=\Big[h_{n-1}(x), \Gamma_*\big(y+\Omega(x)\big)\Big]
\]
where 
\begin{gather*}
\Gamma=\gamma'\circ\gamma,\quad \Omega=\omega+\gamma_*^{-1}\circ\omega'\circ f_{n-1}.
\end{gather*}
Since \eqref{cond} holds for $f$ and $g$, we can show that it holds for the composition as well:
\begin{align*}
\ev^{-1}(h_{n-1}^*\kappa_{n-1}^Z-\Gamma_*\kappa_{n-1}^X)&=f_{n-1}^*\ev^{-1}(g_{n-1}^*\kappa_{n-1}^Z)-
\gamma'_*\ev^{-1}(\gamma_*\kappa_{n-1}^X)\\
&=\gamma'\ev^{-1}(f_{n-1}^*\kappa_{n-1}^Y-\gamma_*\kappa_{n-1}^X)+f_{n-1}^*\ev^{-1}(g_{n-1}^*\kappa_{n-1}^Z-\gamma'_*\kappa_{n-1}^Y)\\
 &=(\gamma'\gamma)_*\delta(\omega)+f_{n-1}^*\gamma'_*\delta(\omega')\\
 &=\Gamma_*\delta(\Omega).
\end{align*}

One can be convinced that the inverse map to $f$ is 
\[f^{-1}(x',y')=(f_{n-1}^{-1}(x'),\gamma_*^{-1}(y'-(f^{-1}_{n-1})^*(\gamma_*\omega(x')))).\]
Again, it can be shown that \eqref{cond} holds for $f^{-1}$ :
\begin{align*}
\ev^{-1}((f_{n-1}^{-1})^*\kappa_{n-1}^X-\gamma^{-1}_*\kappa_{n-1}^Y)
&=(f_{n-1}^{-1})^*\gamma^{-1}_*\ev^{-1}(\gamma_*\kappa_{n-1}^X-f_{n-1}^*\kappa_{n-1}^Y)\\
&=\gamma_*^{-1}\delta((f^{-1}_{n-1})^*(-\gamma_*\omega )).
\end{align*}
\end{proof}

We have seen that Theorem \ref{suf} can be easily reformulated in terms of effective homotopy equivalences. We will do the same with Theorem \ref{nec}. It turns out that if we want to have the necessary condition \eqref{eq} for effective homotopy equivalences, we naturally relax from the strict commutativity in diagram \eqref{d1}. More details are available in the subsequent remark.

\begin{thm}\label{nes2}
Let $X$ and $Y$ be simply connected simplicial sets with effective homology. Let $\{X_{n}, k_{n}^X,\varphi_{n}^X\}$ and $\{Y_{n},
k_{n}^Y,\varphi_{n}^Y \}$ be their effective Postnikov towers with Postnikov cocycles $\kappa_{n-1}^X$, $\kappa_{n-1}^Y$,
respectively. Let $X':=\varprojlim X_n$ and  $Y':=\varprojlim Y_n$. If there is a simplicial weak homotopy equivalence $f\colon X'\to Y'$, then there are effective homotopy equivalences $f_{n}\in\isoef(X_{n}, Y_{n})$ of the Postnikov stages such that the diagrams
\begin{equation}
\begin{tikzpicture}
\matrix (m) [matrix of math nodes, row sep=2em,
column sep=2em, minimum width=2em]
{  X' & Y'  \\
  X_{n} & Y_{n} \\};
  \begin{scope}[every node/.style={scale=.8}]
\path[->](m-1-1) edge node[above] {$f$} (m-1-2);
\path[->](m-1-2) edge node[right] {$\varphi^{Y'}_{n}$} (m-2-2);
\path[->](m-1-1) edge node[left] {$\varphi^{X'}_{n}$} (m-2-1);
\path[->](m-2-1) edge node[below] {$f_{n}$} (m-2-2);
\end{scope}
\end{tikzpicture}
\label{d1}
\end{equation}
commute up to homotopy and the diagrams 

\begin{equation}\label{ee}
\begin{tikzpicture}
\matrix (m) [matrix of math nodes, row sep=2em,
column sep=2em, minimum width=2em]
{  X_{n} & Y_{n}  \\
  X_{n-1} & Y_{n-1} \\};
  \begin{scope}[every node/.style={scale=.8}]
\path[->](m-1-1) edge node[above] {$f_{n}$} (m-1-2);
\path[->>](m-1-2) edge node[right] {$p_{n}^Y$} (m-2-2);
\path[->>](m-1-1) edge node[left] {$p_{n}^X$} (m-2-1);
\path[->](m-2-1) edge node[below] {$f_{n-1}$} (m-2-2);
\end{scope}
\end{tikzpicture}
\end{equation}
commute strictly. 

Furthermore, there is an isomorphism $\gamma\colon \pi_{n}^X\to \pi_{n}^Y$ such that the Postnikov classes  satisfy the relation \eqref{eq}
\begin{equation*}
\gamma_*[{\kappa}_{n-1}^{X}]=f_{n-1}^*[\kappa_{n-1}^Y].
\end{equation*}
\end{thm}

\begin{proof}
In the proof of Theorem \ref{nec} we have shown the existence of (not necessary effective) simplicial weak homotopy equivalences $f'_n\colon X_n\to Y_n$ between Postnikov towers $\{X_{n}\}$ and $\{Y_{n}\}$ together with isomorphisms $\gamma\colon\pi_n^X\to\pi_n^Y$ such that 
\begin{equation}
\gamma_*([\kappa_{n-1}^X])=(f'_{n-1})^*([\kappa_{n-1}^Y])\quad \text{ in } H^{n+1}(X_{n-1};\pi_n^Y).  \label{CE1}  
\end{equation}
The maps $f'_n$ make diagrams \eqref{d1} and \eqref{ee} commutative.

\medskip
By induction, we prove the existence of maps $f_{n}\sim f_n'$ with required properties.
Since the first stages of Postnikov towers are points, $f_1=f_1'$ is the trivial map between points. Now, assume that we have constructed maps $f_m$ for $m\leq n-1$.

Since $f_{n-1}\sim f_{n-1}'$, equation \eqref{CE1} holds for $f_{n-1}$ as well and it enables us to use Theorem \ref{suf} to construct a lift $g_n$ of $f_{n-1}$
of the form \eqref{form}. Let us recall that $f'_n$ is  the lift of $f_{n-1}'$. Denote by 
$\bar g_n,\, \bar f_n'\colon K(\pi_n^X,n)\to K(\pi_n^Y,n)$ the restrictions of $g_n$ and $f_n'$ to the fibres, respectively. Since $\Hom(\pi_n^X,\pi_n^Y)\cong[K(\pi_n^X,n),K(\pi_n^Y,n)]$, homotopy classes of $\bar g_n$ and $\bar f_n'$ are determined by homomorphisms $\pi_n^X\to \pi_n^Y$. According to the proof of Theorem \ref{nec} these homomorphisms are the same as homomorphisms from condition \eqref{eq}
for $f_{n-1}$ and $f_{n-1}'$, respectively. However these are the same, equal to the isomorphism $\gamma$. Consequently, $[\bar g_n]=[\bar f_n']$ which means that the maps
\[\iota^Y\circ\bar{g_n} = g_n\circ\iota^X,\, \iota^Y\circ\bar{f_n'}  = f_n'\circ\iota^X\colon K(\pi_n^X,n)\to \iota^Y (K(\pi_n^Y,n))\]
are homotopic.

In this situation we can apply Addendum \ref{ap} to get  that $f'_n\sim g_{n}+\gamma_*\circ c\circ p_n^X$ for a map $c\colon X_{n-1}\to K(\pi_n^X,n)$. Thus define a new map $f_{n}=g_{n}+\gamma_*\circ c\circ p_n^X$ which is clearly the lift of $f_{n-1}$, is an element of $\isoef(X_n,Y_n)$, has the property $f_{n}\sim f'_n$ and so makes \eqref{d1} commutative up to homotopy.
\end{proof}

\begin{rem}
The new version of the necessary condition can make a weaker impression as one diagram commutes only up to the homotopy. However, algorithms typically work with finite simplicial complexes of dimension $d$, and the input map $f\colon X'\to Y'$ represents a certain homotopy class. Thus, diagram (\ref{d1}) can be made strictly commutative by replacing $f$ with another representative of the same homotopy class using the homotopy lifting property of the fibration $\varphi^{Y'}_{d}$. Such a replacement $H(-,1)$ is a lift in the next diagram with $H'$ to be a homotopy between $\varphi^{Y'}_{d}f$ and $f_{d}\varphi^{X'}_{d}$.
\begin{equation*}
\begin{tikzpicture}
\matrix (m) [matrix of math nodes, row sep=2em,
column sep=2em, minimum width=2em]
{ X'\times\{0\}   & Y'  \\
  X'\times\Delta^1  & Y_{d} \\};
  \begin{scope}[every node/.style={scale=.8}]
\path[->,dashed](m-2-1) edge node[above] {$H$} (m-1-2);
\path[->>](m-1-2) edge node[right] {$\varphi^{Y'}_{d}$} (m-2-2);
\path[->](m-1-1) edge node[above] {$f$} (m-1-2);
\path[>->](m-1-1) edge (m-2-1);
\path[->](m-2-1) edge node[below] {$H'$} (m-2-2);
\end{scope}
\end{tikzpicture}
\end{equation*}
\end{rem}

The second important property of effective homotopy equivalences is that they can approximate every simplicial weak homotopy equivalence up to homotopy.

\begin{prop}\label{iso}
Let $X$ and $Y$ be simply connected simplicial sets with effective homology. Let $\{X_{n}, k_{n}^X, \varphi_{n}^X\}$ and 
$\{Y_{n}, k_{n}^Y,\varphi_{n}^Y\}$  be their effective Postnikov towers. 
Every simplicial weak homotopy equivalence $g_n:X_{n}\to Y_{n}$ is homotopic to an effective homotopy equivalence $f_{n}\in\isoef(X_{n},Y_{n})$. So 
\[\Iso(X_{n},Y_{n})=\Isoef(X_{n},Y_{n}).\]
\end{prop}

\begin{proof}
We will apply Theorem \ref{nes2}. Put  $X'=X_{n}$ and $Y'=Y_{n}$, together with $\varphi_{n}^{X'}=\id$ and $\varphi_{n}^{Y'}=\id$, and $f=g_n:X_n\to Y_n$. Then according Theorem \ref{nes2} there is a an effective homotopy equivalence $f_n:X_{n}\to Y_{n}$ such that  $\varphi_{n}^{Y'}f_n\sim g_n\varphi_{n}^{X'}$ i.e. $f_n\sim g_n$. 

\end{proof}

The next corollary rewrites the necessary and sufficient condition \eqref{eq} using the parame\-tri\-za\-tion of a base map $f_{n-1}$  via self-homotopy equivalences from $\autef(X_{n-1})$ and isomorphisms $\pi_n^X\to \pi_n^Y$. 
 
\begin{cor}\label{cor}
Let $X$ and $Y$ be simply connected simplicial sets with effective homology. Let $\{X_{n}, k_{n}^X, \varphi_{n}^X\}$ and 
$\{Y_{n}, k_{n}^Y,\varphi_{n}^Y\}$ be their effective Postnikov towers with Postnikov cocycles $\kappa_{n}^X$ and $\kappa_{n}^Y$, respectively. Assume that there is a simplicial weak homotopy equivalence $g_{n-1}\colon X_{n-1}\to Y_{n-1}$. Then, $X_{n}$ and $Y_{n}$ are simplicially weak homotopy equivalent if and only if there is an isomorphism $\gamma\colon \pi_{n}^X\to\pi_{n}^Y$ and a simplicial weak homotopy selfequivalence $a_{n-1}\colon X_{n-1}\to X_{n-1}$ such that
\begin{equation}
\gamma_*[{\kappa}_{n-1}^{X}]=(g_{n-1}a_{n-1})^*[\kappa_{n-1}^Y].\label{eqq}
\end{equation}
Moreover, if we assume that $g_{n-1}\in\isoef(X_{n-1},Y_{n-1})$ then $a_{n-1}\in\autef(X_{n-1})$.
\end{cor}

\begin{proof}
Theorem \ref{suf} says that condition \eqref{eqq} is sufficient for an existence of a simplicial weak homotopy equivalence which is a lift of $g_{n-1}\circ a_{n-1}$. If $g_{n-1}$ and $a_{n-1}$ are effective then their composition is effective by Proposition \ref{isoef}.  So the lift is also effective. 

We will show that the condition is necessary. Suppose that $X_n$ and $Y_n$ are simplicially weak homotopy equivalent through a simplicial weak homotopy equivalence $g\colon X_n\to Y_n$. Then, Proposition \ref{iso} provides an effective homotopy equivalence $f_n\in \isoef(X_n,Y_n)$ homotopic to $g$. Since $f_n$ is effective, it is has to be a lift of an effective homotopy equivalence $f_{n-1}\in \isoef(X_{n-1},Y_{n-1})$.
Moreover, there is an isomorphism $\gamma\colon\pi_n^X\to\pi_n^Y$ such that
\begin{equation*}
\gamma_*[{\kappa}_{(n-1)*}^{X}]=(f_{n-1})^*[\kappa_{(n-1)}^Y].
\end{equation*}
Now, it is enough to take $a_{n-1}=g^{-1}_{n-1}\circ f_{n-1}$ to get (\ref{eqq}). 
If $g_{n-1}\in\isoef(X_{n-1},Y_{n-1})$ then $a_{n-1}=g^{-1}_{n-1}\circ f_{n-1}\in\autef(X_{n-1})$ by Proposition \ref{isoef}.

\end{proof}

We will examine the groups $\Autef(X_{n})$ in the following sections. It turns out that they are finitely generated and that there are algorithms that can find their generators inductively. 

\section{Algorithmic group theory}
This section introduces two algorithms for computing orbits and stabilizers of group actions. The aim of this excursion into algorithmic group theory is to give an algorithm that, from the knowledge of   $\Autef(X_{n-1})$ computes  $\Autef(X_{n})$, see Proposition \ref{PF}.

\begin{defn}
Let $G$ be a group. The \emph{right action} of the group $G$ on a set $M$ is a map
\[ R\colon M\times G\to M\]
satisfying $R(m,1)=m$ and $R(m, g\cdot h)=R(R(m,g),h)$ for the unit $1\in G$ and all $m\in M$ and $g,h\in G$.
For simplicity, we will denote the value $R(m,g)\in M$ as $m^g$. The \emph{orbit} of $m\in M$ under $G$ is a subset 
\[m^G:=\{m^g|\,g\in G\}\subseteq M.\]
The \emph{stabilizer} subgroup of $m\in M$ under this action is a subgroup
\[\operatorname{Stab}_G(m):=\{g\in G|\,m^g=m\}\subseteq G.\]
The action is \emph{transitive} if $m^G=M$ for all $m\in M$.
For a subgroup $S$ of $G$, a subset $T\subseteq G$ is called a \emph{right transversal} of $S$ if $T$ contains exactly one element of each coset from 
\[G/S=\{Sg\subseteq G;\ g\in G\}.\] 
\end{defn}

\begin{exmp}
We will use these notions in the following context. Take a Postnikov tower $\{X_{n}\}$ for a finite simplicial set $X$. It can be easily verified that the map
\begin{align*}
\Tor(H^{n+1}(X_{n-1};\pi_n))\times\big(\Autef(X_{n-1})\times\Aut(\pi_n)\big)&\to \Tor(H^{n+1}(X_{n-1};\pi_n))\\
([\kappa],[a_{n-1}],\gamma)&\longmapsto \gamma^{-1}_*a_{n-1}^*[\kappa]
\end{align*}
is a right action of the group $G=\Autef(X_{n-1})\times\Aut(\pi_n)$ 
on the finite set 
\[M=\Tor(H^{n+1}(X_{n-1};\pi_n)).\]
\end{exmp}

\begin{lem}[Orbit-Stabilizer]
Let $G$ acts on a set $M$, $m\in M$ and let $S:=\operatorname{Stab}_G(m)$. The map
\[  G/S\longrightarrow m^G:\ Sg\longmapsto m^g \]
 is a well-defined bijection. 
\end{lem}

\begin{proof}
See \cite[Theorem 2.16]{HEB}.
\end{proof}

The next lemma constitutes a way to construct a presentation of a subgroup. Consider a subgroup $S$ of a group $G$ and a right transversal $T$ of $S$. The uniquely determined element in $T\cap Sg$ will be denoted $\overline{g}$.
\begin{lem}[Schreier's lemma]\label{SL}
Let $G$ be a group generated by elements from a set $P$, $G=\langle P\rangle$. Let  $S$ be a subgroup of $G$ with a finite right transversal  $T$. Then the subgroup $S$ is generated by
the set
\begin{equation}\label{SG}
S\cap\{rp(\overline{rp})^{-1}\in G;\ r\in T,\, p\in P, rp\ne\overline{rp} \}.
\end{equation}
\end{lem}

\begin{proof}
See \cite[Theorem 2.57]{HEB}.
\end{proof}

As a direct consequence of these two assertions on stabilizers, we get.

\begin{cor}\label{FS}
Let a finitely generated group $G$ act on a finite set $M$, $m\in M$ and let $S:=\operatorname{Stab}_G(m)$. 
Then the subgroup $S$ is finitely generated, and its generators are listed in the set \eqref{SG}.
\end{cor}

\begin{proof}
The Orbit-Stabilizer theorem provides a map that identifies $S$ as a finite index subgroup of $G$. Thus, any right transversal of $S$ in $G$ is a finite set, and the generators of $S$ are enumerated by Schreier's lemma.
The list of generators is finite as both sets $P$ and $T$ are finite.
\end{proof}

The initial theory easily gives two algorithms. The first one computes the orbit $m^G$, and its extended version enhances the output by a list of generators of the stabilizer $\operatorname{Stab}_G(m)$. We start with a specification of input objects. 

\begin{con}
Let $G$ be a finitely-generated group defined by an explicit finite generating set $P$. Take a finite set $M$ endowed with a right $G$-action such that:
\begin{itemize}
    \item $m^g$ can be computed for all $m\in M$ and $g\in G$,
    \item one can decide whether two elements of $M$ are equal. 
\end{itemize}
\end{con}

If $G$ is a finitely-generated group and $M$ is a finite set, then any orbit can be computed by an exhaustive search.\\

\begin{algorithm}[H]
\caption{Orbit algorithm}\label{O}
\SetAlgoLined
\KwData{A finitely-generated group $G=\langle P\rangle$ with a right action on a finite set $M$, an element $m\in M$.}
\KwResult{Elements of the orbit $m^G$.}
List:=$\{m\}$\;

  \ForEach{ $y\in$ List }{
  \ForEach{$g\in P$}{
   \If{ $y^g\not\in$ List}{
      List.Append($y^g$)\;
   }
  }
  }

\Return List\;
\end{algorithm}
\smallskip

Now we want to find a list of generators of the stabilizer $S=\operatorname{Stab}_G(m)$. Since $G/S$ is in bijection with the orbit $m^G$, for every $y\in m^G$ we choose a single $g\in G$ such that $m^g=y$. For this choice, we will write $g=\log y$:
\[m^{\log y}=y.\]
The set
\[ T=\{\log y\in G;\ y\in m^G\}\]
is a finite right transversal to the stabilizer  $\operatorname{Stab}_G(m)$. For all $y\in m^G$ and all $g\in G$ can write 
\[m^{\log(y^g)}=y^g=(m^{\log y})^g=m^{(\log y)\cdot g}.\]
Consequently, comparing our previous and new notation we get
\[\overline{(\log y)\cdot g}=\log (y^g).\]
Hence, in the case of the stabilizer $\operatorname{Stab}_G(m)$ the set of Schreier's generators from Lemma \ref{SL} can be parametrized by the elements of the orbit and the elements of the set $P$ and is
\[Q=\{(\log y)g(\log({y^g}))^{-1}\in G;\ y\in m^G,\ g\in P\}.\]
Note that every element of $Q$ lies in $\operatorname{Stab}_G(m)$ as:
\[m^{(\log y)g(\log({y^g}))^{-1}}=[m^{(\log y)g}]^{(\log({y^g}))^{-1}}=(y^g)^{(\log({y^g}))^{-1}}=m^{\log({y^g})(\log({y^g}))^{-1}}=m.\]
This enables us to extend the previous algorithm for computing the orbit to the algorithm computing  simultaneously the orbit $m^G$ and generators of the stabilizer
$\operatorname{Stab}_G(m)$.\\

\begin{algorithm}[H]\label{OS}
\caption{Orbit-Stabilizer algorithm}
\SetAlgoLined
\KwData{A finitely-generated group $G=\langle P\rangle$ with a right action on a finite set $M$, an element $m\in M$.}
\KwResult{Elements of the orbit $m^G$ and a set $Q$ of Schreier's generators  of the stabilizer $\operatorname{Stab}_G(m)$.}
$\log m:=1$\;
List:=$\{(m,\log m)\}\subseteq M\times G$\; 
Q:=$\{\}$\;
  \ForEach{ $(y,\log y)\in$ List }{
  \ForEach{$g\in P$}{
   \eIf{ $y^g\not\in$ GetFirstComponents(List)}{
    $\log(y^g):=\log(y)g$\; 
     List.Append($(y^g,\log(y^g)$)\;
   }{
    Q.Append($\log(y)g(\log y^g)^{-1}$)\;
   }
  }
  }

\Return List, Q\;
\end{algorithm}

\vskip 2mm

We will apply the last algorithm in the situation described in the example above. However, we specify the input data in a more precise manner. 
Consider $M$ as the torsion part of the cohomology of $X_{n-1}$ computed from an effective chain complex, i.e.  
\[M=\Tor(H^{n+1}(\EC_*(X_{n-1});\pi_n)).\]
The set $M$ is finite by definition and computable as $H^{n+1}(\EC_*(X_{n-1});\pi_n)$ is the effective cohomology, and its torsion component is effectively computable as well. We suppose that the group $\Autef(X_{n-1})$ is finitely generated, and its generators are effectively computable. 
The right action of the group $G=\Autef(X_{n-1})\times \Aut(\pi_n)$ on $M$ has been described in the example after Definition 5.1. In the proof of the key statement \ref{PF} we will use the second algorithm to compute generators of the stabilizer  $\operatorname{Stab}_G([\kappa^{\operatorname{ef}}_{n-1}])$ and its right transversal.

\section{Spaces of finite $k$-type}
In this section, we will deal with generators of the groups $\Autef(X_{n})$. We would like to have an algorithm which from the knowledge of a finite list of generators for $\Autef(X_{n-1})$ gives a finite list of generators for  $\Autef(X_{n})$. The necessary and sufficient condition \eqref{eq} for lifting homotopy selfequivalences
\[\alpha_*[\kappa_{n-1}^X]=a_{n-1}^*[\kappa_{n-1}^X]\]
defines a stabilizer of an element $[\kappa_{n-1}^X]$ of the action of the group  $\Autef(X_{n-1})\times \Aut(\pi_n^X)$ on the group $H^{n+1}(X_{n-1};\ \pi_n^X)$. In general, this group is an infinite set and in such a case it is difficult to find an algorithm computing a list of generators for  $\Autef(X_{n})$. However, the situation becomes much easier if the Postnikov class $[\kappa_{n-1}^X]$ is of finite order. It happens for so-called spaces of finite $k$-type, see  \cite{AC}, Definition 4.4(a). We will modify this definition for our purposes.

\begin{defn}
Let $X$ be a simply connected simplicial set with effective homology. The set $X$ has finite $k$-type if its effective Postnikov tower $\{X_n,\kappa_n^X\}$ has Postnikov class $[\kappa_n^X]$ of finite order for each $n\in\Nbb$. 

We will say that $X$ has finite $k$-type through dimension $d$ if Postnikov classes
$[\kappa_n^X]$ are of finite order for all $n\le d-1$.
\end{defn}

The notion of space with finite $k$-type encompasses a broad class of spaces, for instance, $H$-spaces modulo the class of finite abelian groups (see \cite{AC}, Definition 4.4(b)).

\begin{defn}
A simply connected simplicial set  $X$ is an $H$-space modulo the class of finite abelian groups $\mathcal C$ if
there is a map $\mu\colon |X|\times |X|\to |X|$ such that the compositions $\mu\circ \iota_i$ are isomorphisms in homology modulo the class $\mathcal C$ where the maps $\iota_i\colon |X|\to |X|\times |X|$ are obvious inclusions to the $i$-th component of the product $|X|\times |X|$, $i=1,2$. 

If Postnikov stage $X_d$ is an $H$-space modulo the class $\mathcal C$, we will say that $X$ is an 
$H$-space modulo $\mathcal C$ through dimension $d$.
\end{defn}

In \cite{AC} Theorem 4.5, Arkowitz and Curjel proved the following statement:

\begin{prop}\label{hs}
Every $H$-space mod $\mathcal C$ has a finite $k$-type. Moreover, if $H_i(X)=0$ for all $i$ sufficiently large, then $X$ is an $H$-space mod $\mathcal C$ if and only if it is a space of finite $k$-type. 
\end{prop}

We can adapt this result to our situation:

\begin{cor}\label{hsd}
Let $X$ be a finite simply connected simplicial set of dimension $d$. It is an $H$-space mod $\mathcal C$ through dimension $d$ if and only if it is of finite $k$-type through dimension $d$. 
\end{cor}

\begin{rem}
It is well known that every $m$-connected space, $m\geq 1$, is a rational $H$-space through dimension 
$2m$ and an $H$-space mod $\mathcal C$ through the same dimension. Consequently, it is also of finite $k$-type through dimension $2m$. 
\end{rem}

\subsection*{Homotopy self-equivalence of $X_{n}$}
Let $X$ be a space of finite $k$-type and $\{X_{n}\}$ its effective Postnikov tower. Here our attention focuses on studying  $\Autef(X_{n})$. 
We follow the approach of the previous section and utilise Orbit-Stabilizer algorithm on the finite order element $[\kappa_{n-1}^X]$ with the right action
\begin{align}
\Tor(H^{n+1}(X_{n-1};\pi_n^X))\times(\Autef(X_{n-1})\times \Aut(\pi_n^X))&\to \Tor(H^{n+1}(X_{n-1};\pi_n^X))\notag\\
([\kappa],[a], \gamma)&\longmapsto \gamma^{-1}_*a^*[\kappa].\label{FreeAct}
\end{align}
Note that the group $G=\Autef(X_{n-1})\times \Aut(\pi_n^X)$ is finitely generated whenever $\Autef(X_{n-1})$ has this property. Moreover, its generators  can be taken as the pairs $([a],e)$ and $([\operatorname{id}],\gamma)$ where $[a]$ goes through all generators of $\Autef(X_{n-1})$, $\gamma$ runs through generators of 
$\Aut(\pi_n^X)$ and $e\in \Aut(\pi_n^X)$ is the neutral element. This choice is possible due to the commutativity of the action of $\gamma_*^{-1}$ and $a^*$ 
on $\Tor(H^{n+1}(X_{n-1};\pi_n^X))$.
The torsion $\Tor(H^{n+1}(X_{n-1};\pi_n^X))$ is computable and finite, see the related discussion at the end of Section 6. 

\begin{prop}\label{PF}
Let $X$ be simply connected finite simplicial set of finite $k$-type through dimension $d$. Consider its Postnikov tower $\{X_{n}\}$. 
If $n \leq d$, generators of $\Autef(X_{n-1})$ are known and form a finite list, then there is an algorithm that computes a finite list of generators of $\Autef(X_{n})$.
\end{prop}

\begin{proof}
First we describe the algorithm and then we show that it really gives all the generators of $\Autef(X_{n})$. Since the group  $\Autef(X_{n-1})$ is finitely generated and $[\kappa^X_{n-1}]$ lies in a finite set $\Tor(H^{n+1}(X_{n-1};\pi_n^X))$, we can use Orbit-Stabilizer algorithm to compute all the generators
of the stabilizer $\operatorname{Stab}_G([\kappa^X_{n-1}])$. Denote them $([a_j],\gamma_j)$, $j=1,2,\dots, k$. Indeed, they satisfy the relation
$$ \gamma_{j*}^{-1}a_j^*[\kappa^X_{n-1}]=[\kappa^X_{n-1}],$$
hence we can find algorithmically maps $\omega_j\colon X_{n-1}\to E(\pi_n^X,n)$ such that 
\[\ev^{-1}(\gamma_{j*}^{-1}a_j^*\kappa^X_{n-1}-\kappa^X_{n-1})=\delta\omega_j.\]
Then the homotopy classes of self-maps 
 \[h_j(x,y)=(a_j(x),\gamma_{j*}(y + \omega_j(x)))\]
are elements of  $\autef(X_{n})$. 

Next we take the cochains $c_1,c_2,\dots, c_p$ such that their cohomology classes are generators of the group $H^n(X_{n-1};\pi_n^X)$. The translations $t_m:X_n\to X_n$
\[t_m(x,y)=(x,y+c_m(x))\]
are elements of  $\autef(X_{n})$, as well. We assert that the homotopy classes $[h_j]$, $1\le j\le k$, 
and $[t_m]$, $1\le m\le p$, form a set of generators of the group
$\Autef(X_{n})$.

\medskip

Take a class $[A]\in \Autef(X_{n})$ whose representative $A$ has the form
    \[A(x,y)=(a(x),\gamma_*(y+\omega(x)))\]
where $\gamma\in\Aut(\pi_{n}^X)$, $a\in\autef(X_{n-1})$ and $\omega\colon X_{n-1}\to E(\pi_n^X,n)$  satisfy 
\[\ev^{-1}((\gamma)_*^{-1}a^*\kappa^X_{n-1}-\kappa^X_{n-1})=\delta\omega.\] 
Hence the pair $([a],\gamma)$ is an element of the stabilizer $\operatorname{Stab}_G([\kappa^X_{n-1}])$. Thus we can write it as a product of the generators $([a_j],\gamma_j)$:
\[([a],\gamma)=([a_{j_1}],\gamma_{j_1})\cdot ([a_{j_2}],\gamma_{j_2})\dots \cdot([a_{j_s}],\gamma_{j_s})\]
where some indices can repeat. The composition 
$h= h_{j_1}\circ h_{j_2}\dots\circ h_{j_s}$
has the form
\[h(x,y)=(a(x), \gamma_*(x+\Omega (y)))\]
where
\[\Omega = \omega_{j_s}+(\gamma_{j_s}^{-1})_*\omega_{j_{s-1}}a_{j_s}+\dots+(\gamma_{j_2}\dots \gamma_{j_{s-1}}\gamma_{j_{s}})^{-1}_*\omega_{j_1}a_{j_2}\dots a_{j_s}\]
satisfies
\[\ev^{-1}((\gamma)_*^{-1}a^*\kappa^X_{n-1}-\kappa^X_{n-1})=\delta\Omega.\]
Since $\delta\omega=\delta\Omega$, there is $c\colon X_{n-1}\to K(\pi_n^X,n)$ such that
$\omega=\Omega+c$.
If we define the translation $t\colon X_n\to X_n$, $t(x,y)=(x, y+c(x))$, we get
\[A=h\circ t\]
The cohomology class $[c]$ is a linear combination of the cohomology classes $[c_m]$, $1\le m\le p$, hence the translation  $t$ is homotopic to the corresponding composition of translations $t_m$. Consequently, the homotopy selfequivalence $A$ is homotopic to a composion of homopy selfequivalences $h_j$, $1\le j\le k$, and $t_m$, $1\le j\le p$.  

\end{proof}

\section{Proof of the main result}
The last section presents a sequence of algorithms that decides whether two spaces of finite $k$-type are simplicially weak homotopy equivalent. It is a concatenation of the results from the previous four sections.

The following Theorem is a more detailed version of Theorem A.

\begin{thm}\label{main}
Let $X$ and $Y$ be simply connected finite simplicial sets of dimension $d$. If $X$ and $Y$ are of finite $k$-type through dimension $d$, then the question of whether $X$ and $Y$ are simplicially weak homotopy equivalent is algorithmically decidable. 

If the spaces are simplicially weak homotopy equivalent, our algorithm finds a simplicial weak homotopy equivalence $f\colon X_d\to Y_d$ between their Postnikov stages in dimension $d$. 

Moreover, we can algorithmically find all group generators $\Autef(X_d)$. If $X$ and $Y$ are simplicially weak homotopy equivalent, we can consequently describe all simplicial weak homotopy equivalences between $X_d$ and $Y_d$ up to homotopy as compositions of $f$ and generators of $\Autef(X_d)$.
\end{thm}

We describe the proof of Theorem \ref{main} in single steps.

\subsection*{Pre-processing} First, we compute effective Postnikov towers for both sets $X$ and $Y$ according to the instructions in \cite[Section 4]{poly}. That algorithm provides suitable isomorphism types $\pi_n^X$ and $\pi_n^Y$ of homotopy groups and gives a fixed decomposition of these groups to the components $\Free$ and $\Tor$ as a part of their isomorphism type.
If the homotopy groups $\pi_n^X$ and $\pi_n^Y$ have different isomorphism types, $X$ and $Y$ are definitely not simplicially weak homotopy equivalent.
Simultaneously, by computing Postnikov classes, one can be convinced that both spaces are of finite $k$-type through dimension $d$.
\smallskip

Inductively, we compute generators of the groups $\Autef(X_{n})$ for $n\le d$ using the following algorithm.

\begin{alg} $\ $
\newline
{\rm Input:} A finite set of generators of the group  $\Autef(X_{n-1})$ and $[\kappa_{n-1}^X]$ is torsion in $H^{n+1}(X_{n-1};\pi_n^X)$.
\newline
{\rm Output:} A finite set of generators of the group  $\Autef(X_{n})$.
\end{alg} 

\begin{proof}[Description] A complete list of generators of $\Autef(X_{n})$ is provided in Proposition \ref{PF}. Every such generator represents a simplicial weak homotopy equivalence in $\autef(X_{n})$.
A key part of that procedure is Orbit-Stabilizer algorithm and known description of generators of the group
$\Aut(\pi_n^X)\cong\Aut(\Free(\pi_n^X)\oplus\Tor(\pi_n^X))$ depending on
$\Aut(\Free(\pi_n^X))\cong\operatorname{GL}(\operatorname{rank}(\pi_n^X), \mathbb Z)$, $\Aut(\Tor(\pi_n^X))$ and $\Hom(\Free(\pi_n^X), \Tor(\pi_n^X))$.
\end{proof}

\begin{rem}
The groups $\Aut(\Free(\pi_n^X))$ can be identified with $\GL(\Zbb,m_n)$. It is well known that the minimal number of generators of $GL(\Zbb,m)$ is 2. We refer to \cite[Chapter VII]{MN} for more details about the generators.
\end{rem}

In Corollary \ref{criterion}, it has been shown that the spaces $X$ and $Y$ of dimension $d$ are simplicially weak homotopy equivalent if and only if their Postnikov stages $X_d$ and $Y_d$ are simplicially weak homotopy equivalent. We will subsequently decide if their Postnikov stages are simplicially weak homotopy equivalent by going from $X_1$ and $Y_1$ to $X_d$ and $Y_d$.

\subsection*{Base step} Since $X_1=Y_1$ is a point, there is only one simplicial weak homotopy equivalence between them, so $\isoef(X_{1},Y_{1})=\{\id\}$.
 
\subsection*{Induction step} This step means to find whether there is a simplicial weak homotopy equivalence 
\[f_{n}\in\isoef(X_{n},Y_{n})\]
under the condition that we have a simplicial weak homotopy  equivalence $f_{n-1}\in \isoef(X_{n-1},Y_{n-1})$. This is carried out for $2\le n\le d$ by the following algorithm whose main idea is to seek  $a\in\aut^{\operatorname{ef}}(X_{n-1})$ fitting into the diagram:

\begin{center}
\begin{tikzpicture}
\matrix (m) [matrix of math nodes, row sep=2em,
column sep=4em, minimum width=2em]
{ 
X_{n} & {} & Y_{n} \\
X_{n-1} & X_{n-1}  & Y_{n-1}  \\};
  \begin{scope}[every node/.style={scale=.8}]
\path[->,dashed](m-2-1) edge node[below] {$a$} (m-2-2);
\path[->](m-2-2) edge node[below] {$f_{n-1}$} (m-2-3);
\path[->](m-1-1) edge node[auto] {$f_{n}$}  (m-1-3);
\path[->>](m-1-1) edge (m-2-1);
\path[->>](m-1-3) edge (m-2-3);
\path[->](m-2-1) edge[bend right=30] node[below] {$g_{n-1}$} (m-2-3);
\end{scope}
\end{tikzpicture}
\end{center}
If the simplicial weak homotopy equivalence $a$ exists then we update all $f_i$ to the maps $f_i\circ a_i$ for each $2\leq i\leq n-1$. More implementation details are part of the next algorithm.
\begin{alg} $\ $
\newline
{\rm Input:} A simplicial weak homotopy equivalence $f_{n-1}\in\isoef(X_{n-1},Y_{n-1})$ and Postnikov classes $[\kappa_{n-1}^X]$ and $[\kappa_{n-1}^Y]$ in the torsion parts of $H^{n+1}(X_{n-1};\pi_n^X)$ and $H^{n+1}(Y_{n-1};\pi_n^Y)$, respectively. 
\newline
{\rm Output:} A decision whether $X_{n}$ and $Y_{n}$ are simplicially weak homotopy equivalent. If yes, the algorithm provides a simplicial weak homotopy equivalence  $f_{n}\in \isoef(X_{n},Y_{n})$. If no, then $X$ and $Y$ are not simplicially weak homotopy equivalent.
\end{alg} 

\begin{proof}[Description]
We will use the necessary and sufficient condition from Corollary \ref{cor}. So we want to find if there is a pair $(a,\alpha)\in\autef(X_{n-1})\times\Aut(\pi_n^Y)$ such that
\begin{equation}\label{ns}
\sigma_*[{\kappa}_{n-1}^{X}]=\alpha^{-1}_*a^*f_{n-1}^*[\kappa_{n-1}^Y]
\end{equation}
for a fixed isomorphism $\sigma\colon\pi_n^X\to\pi_n^Y$. As both classes $[\kappa_{n-1}^X]$ and $[\kappa_{n-1}^Y]$ are torsion, we solve the equation by considering the right action of the group $G=\Autef(X_{n-1})\times \Aut(\pi_n^Y)$ from section 6:
\begin{align*}
\Tor(H^{n+1}(X_{n-1};\pi_n^Y))\times G&\to \Tor(H^{n+1}(X_{n-1};\pi_n^Y))\notag\\
([\kappa],[g], \gamma)&\longmapsto \gamma^{-1}_*g^*[\kappa].
\end{align*}
We can apply Orbit-Stabilizer algorithm \ref{OS} on this action and compute the orbit of the element
$$f_{n-1}^*[\kappa_{n-1}^Y].$$ 
Now it suffices to go through the finite list of the orbit elements and to decide if $\sigma_*[{\kappa}_{n-1}^{X}]\in \left(f_{n-1}^*[\kappa_{n-1}^Y]\right)^G$.
If it is missing in the list, then $X_{n}$ and $Y_{n}$ are not simplicially weak homotopy equivalent as the necessary condition (\ref{ns}) is not satisfied. It implies that $X$ and $Y$ are not simplicially weak homotopy equivalent.

If $\sigma_*[{\kappa}_{n-1}^{X}]\in \left(f_{n-1}^*[\kappa_{n-1}^Y]\right)^G$, we take the right transversal $\log(\sigma_*[{\kappa}_{n-1}^{X}])=([a],\alpha)$ and it satisfies  sufficient condition (\ref{ns}) so that $f_{n}$ can be constructed from $g_{n-1}=f_{n-1}\circ a$ by Theorem \ref{suf}.
\end{proof}

\begin{proof}[Proof of Corollary B]
Let $X$ and $Y$ be finite simplicial sets. The dimensions of multiple suspensions $S^rX$ and $S^rY$ of $X$ and $Y$ will be less than twice their connectivity for $r$ sufficiently large. According to the remark after Corollary \ref{hsd}, $S^rX$ and $S^rY$ will be of finite $k$-type through their dimension and will satisfy assumptions of Theorem \ref{main}. So we can apply the described algorithm to decide if $S^rX$ and $S^rY$ are homotopy equivalent, i.e., if $X$ and $Y$ are stably homotopy equivalent. 
\end{proof}
\section{Another class of spaces}
What we have done so far enables us to decide about homotopy equivalence inside another class of spaces which expands spaces with finite homotopy groups differently than spaces of finite $k$-type.

\begin{prop}\label{another}
Let $X$ and $Y$ be a finite simply connected simlicial sets of dimension $d$ such that for $2\le n\le d$
\begin{align*}
\rank\pi_n^X&=\rank\pi_n^Y\le 1,\\
H^n(X_{n-1};\pi_n^X)&\cong\Tor(H^n(X_{n-1};\pi_n^X)).
\end{align*}
Then there is an algorithm that decides if $X$ and $Y$ are simplicially weak homotopy equivalent.
\end{prop}

\begin{proof}
As in the previous section  we compute effective Postnikov towers and we proceed by induction.
The induction step requires to have pre-computed the group $\Autef(X_{n-1})$, and we deal with it in the subsequent rows. 
Since $\rank\pi_n^X=\rank\pi_n^Y\le 1$, the groups $\Aut(\pi_n^X)$ are finite. Now, we can go through the proof of Proposition \ref{PF} to show that if $\Autef(X_{n-1})$ is finite then so is $\Autef(X_{n})$.
If $\Autef(X_{n-1})$ is finite then the subgroup $\Stab([\kappa_{n-1}^X])\subseteq \Autef(X_{n-1})\times \Aut(\pi_n^X)$ is also finite and can be found by exhaustive search. All elements of the group $\Autef(X_{n})$ are parameterized by elements from the stabilizer and elements from the group $H^n(X_{n-1};\pi_n^X)$ which is finite according to our assumption. So $\Autef(X_{n})$ is finite and its elements are enumerated in Proposition \ref{PF} as well. As the last part of Induction step, we use Corollary \ref{cor} to decide which pairs $([a_{n-1}],\gamma)\in \Autef(X_{n-1})\times \Aut(\pi_n^X)$ satisfy the necessary and sufficient condition \eqref{eqq} by simply going through all cases. If $n=d$ and the required pair exists then $X$ and $Y$ are homotopy equivalent by Corollary \ref{criterion}. If the required pair is missing for a certain $2\leq n\leq d$ then $X$ and $Y$ are not homotopy equivalent.
\end{proof}

If $X$ is a finite dimensional simplicial set, the finiteness of $H^n(X_{n-1};\pi_n^X)$  can be easily verified via the effective homology framework. 
One concrete example of spaces satisfying this property is co$H$-spaces modulo the class of finite abelian groups. 

\begin{defn}
A simply connected simplicial set $X$ is a co$H$-space modulo the class of finite abelian groups $C$ if
there is a map $\phi\colon |X|\to |X|\vee |X|$ such that the compositions $r_j\circ\phi$ are isomorphisms in homology modulo the class $C$. The maps $r_j\colon |X|\vee |X|\to |X|$ are obvious retractions  on the $j$-th component of the wedge product $|X|\vee |X|$. 
\end{defn}

For CW-complex $Y$, the property to be a co$H$-space is equivalent to the property that $Y$ is a union of two open subsets which are contractible in $Y$, that is, the Lusternik-Schnirelmann category of $Y$ is $\le 2$ in original terminology and $<2$ in current terminology. 
That is why in \cite{BA} co$H$-space modulo the class of finite abelian groups is called a space with the Lusternik-Schnirelmann $\le 2$ modulo finite abelian groups. 

\begin{lem} Let a simply connected set $X$ be a co$H$-space modulo the class of finite abelian groups
with an effective Postnikov tower $\{X_n\}$.
Then the groups $H^n(X_{n-1};\pi_n^X)$ are finite.
\end{lem}

\begin{proof}
Berstein \cite[Remark 3.1]{BA} has shown that finite co$H$-space $X$ modulo the class of finite abelian groups is homotopy equivalent modulo finite abelian groups with a wedge of spheres and consequently has the property that $\operatorname{coker} h_n$ of the Hurewicz homomorphism $h_n\colon\pi_n(X)\to H_n(X)$ is finite. (See also \cite{AC}.)

We will prove that it implies that $H^n(X_{n-1};\pi_n^X)$ is also finite. Consider the Serre long exact cohomology sequence of the fibration $K(\pi_n^X,n)\xrightarrow{i} X_n\to X_{n-1}$ with coefficients $\pi_n^X\otimes\Qbb$ which we supress:
\[\dots\to \underbrace{H^{n-1}(K(\pi_n^X,n))}_{0}\to H^{n}(X_{n-1}) \to H^{n}(X_n)\xrightarrow{i^*} H^{n}(K(\pi_n^X,n))\to  H^{n+1}(X_{n-1})\to\dots\]
If we prove that the morphism $i^*$ is a monomorphism, we get that  $H^{n}(X_{n-1};\pi_n^X\otimes\Qbb)$ is trivial 
and $H^{n}(X_{n-1};\pi_n^X)$ is finite.

The fibration $X\to X_n$ has an $n$-connected fiber and hence $H_n(X)\cong H_n(X_n)$. Then the cokernel of the Hurewicz map 
$h_n^{X_n}\colon \pi_n(X_n)\to H_n(X_n)$ is isomorphic to the cokernel of the Hurewicz map $h_n$, and so it is finite. 
From the diagram
\begin{center}
\begin{tikzpicture}
\matrix (m) [matrix of math nodes, row sep=2em,
column sep=2em, minimum width=2em]
{ H_n(K(\pi_n^X,n)) & H_n(X_n) \\
  \pi_n(K(\pi_n^X,n)) & \pi_n(X_n) \\
};
  \begin{scope}[every node/.style={scale=.8}]
\path[<-](m-1-2) edge node[right] {$h_n^{X_n}$} (m-2-2);
\path[->](m-1-1) edge node[above] {$i_*$} (m-1-2);
\path[<-](m-1-1) edge node[left] {$h_n^{K},\cong$} (m-2-1);
\path[->](m-2-1) edge node[below] {$\cong$}  (m-2-2);
\end{scope}
\end{tikzpicture}
\end{center}
we get that $\operatorname{coker} i_*$ is also finite. Applying $\operatorname{Hom}(-,\pi_n^X\otimes\Qbb)$ and Universal Coefficient Theorem we obtain the exact sequence 
$0\to H^n(X_n;\pi_n^X\otimes\Qbb)\xrightarrow{i^*} H^n(K(\pi_n^X,n);\pi_n^X\otimes\Qbb)$,
which shows that $i^*$ is a monomorphism.
\end{proof}

\subsection*{Acknowledgements}
I would like to thank my supervisor Martin Čadek for proposing this problem, for his guidance, and for many valuable suggestions 
and comments throughout  all the drafts of this paper. I would also like to express my gratitude to Lukáš Vokřínek
for helpful discussions and notable ideas.
Furthermore, the author is grateful to anonymous referees for their significant  comments and suggestions.

\end{document}